\documentclass[a4paper,bibliography=totoc]{scrartcl}

\usepackage{amsmath,amssymb,amsthm}
\usepackage[T1]{fontenc}
\usepackage[utf8]{inputenc}
\usepackage{lmodern}
\usepackage{graphicx}
\usepackage{float}
\usepackage[dvipsnames,svgnames]{xcolor}
\usepackage{tikz,pgf}
\usetikzlibrary{calc}
\usetikzlibrary{patterns}
\usetikzlibrary{backgrounds}
\usetikzlibrary{arrows.meta}
\usetikzlibrary{decorations.pathmorphing}
\usetikzlibrary{decorations.pathreplacing}
\usetikzlibrary{fadings}
\usetikzlibrary{spy}
\usetikzlibrary{shapes.symbols}
\usetikzlibrary{shadows}
\usetikzlibrary{intersections}
\usetikzlibrary{angles,quotes}
\usepackage{booktabs}
\usepackage[colorlinks=true,linkcolor=RoyalBlue,citecolor=PineGreen,urlcolor=RoyalBlue]{hyperref}
\usepackage{caption}
\usepackage{subcaption}
\usepackage[nameinlink]{cleveref}
\usepackage{fontawesome5}

\usepackage{todonotes}

\newtheorem{theorem}{Theorem}[section]
\newtheorem{proposition}[theorem]{Proposition}
\newtheorem{lemma}[theorem]{Lemma}
\newtheorem{corollary}[theorem]{Corollary}

\theoremstyle{definition}
\newtheorem{definition}[theorem]{Definition}
\newtheorem{example}[theorem]{Example}

\newtheorem{remark}[theorem]{Remark}

\newcommand{\range}[2]{\{#1,\ldots,#2\}}
\newcommand{\oriented}[2]{(#1,#2)}
\newcommand{\sage}{\textsc{SageMath}}
\newcommand{\FlexRiLoG}{\textsc{FlexRiLoG}}

\newcommand{\eqwithreference}[1]{\underset{\text{\scriptsize\mbox{#1}}}{=}}

\newcommand{\blue}{\text{blue}}
\newcommand{\red}{\text{red}}
\newcommand{\gold}{\text{gold}}
\newcommand{\conj}[1]{\overline{#1}}
\newcommand{\pairsGold}[1]{\operatorname{gp}_G(#1)}
\newcommand{\clGold}[1]{\operatorname{cl}_\gold(#1)}

\newcommand{\Wfun}[2]{W_{#1,#2}}
\newcommand{\Zfun}[2]{Z_{#1,#2}}

\newcommand{\motion}{\mathcal{M}}
\newcommand{\RSconfigSpace}{\mathcal{V}_s}

\newcommand{\rot}[1]{\Theta_{#1}}
\newcommand{\rott}{\rot{t}}
\newcommand{\rotmt}{\rot{-t}}
\newcommand{\rots}{\rot{s}}
\newcommand{\rotms}{\rot{-s}}

\newcommand{\tw}{\tilde{w}}
\newcommand{\tdelta}{\tilde{\delta}}

\newcommand{\NN}{\mathbb{N}}
\newcommand{\RR}{\mathbb{R}}
\newcommand{\CC}{\mathbb{C}}
\newcommand{\QQ}{\mathbb{Q}}
\newcommand{\ci}{i}

\newcommand{\Trel}{\triangle}
\newcommand{\Prel}{\square}

\definecolor{colR}{HTML}{CC6677}
\definecolor{colB}{HTML}{6699CC}
\colorlet{colG}{Gold!90!black}
\colorlet{colGr}{DarkSeaGreen}
\colorlet{colbg}{white}
\colorlet{colfg}{black}
\colorlet{colgraphv}{colfg!75!colbg}
\colorlet{colgraphe}{colfg!65!colbg}
\colorlet{colline}{colfg!50!colbg}
\colorlet{ecol}{black!50!white}

\tikzstyle{sym}=[ecol,dashed]

\tikzstyle{gvertex}=[circle, draw=colbg, fill=colgraphv, inner sep=0pt, minimum size=4pt]
\tikzstyle{gvertexh}=[rectangle, draw=colbg, fill=colgraphv, inner sep=0pt, minimum size=4pt]

\tikzstyle{fvertex}=[circle,inner sep=0pt,minimum size=2.5pt,fill=colbg!80!colfg,draw=colgraphv,line width=1pt,outer sep=1pt]
\tikzstyle{fvertexh}=[rectangle,inner sep=0pt,minimum size=2.5pt,fill=colbg!80!colfg,draw=colgraphv,line width=1pt,outer sep=1pt]

\tikzstyle{edge}=[line width=1.5pt,colgraphe]
\tikzstyle{dedge}=[edge,-{Latex[width=3.45pt,length=5pt]}] 
\tikzstyle{redge}=[edge,colR]
\tikzstyle{bedge}=[edge,colB]
\tikzstyle{gedge}=[edge,colG]

\tikzstyle{hedge2}=[line width=2.5pt]
\tikzstyle{ledge}=[edge,opacity=0.25]

\tikzstyle{axes}=[colline,-latex]
\tikzstyle{gridl}=[colgraphe,line width=0.5pt]
\tikzstyle{gridl2}=[colgraphe,line width=0.1pt]
\tikzstyle{fvertexs}=[fvertex,minimum size=0.18pt,line width=0.05pt,outer sep=0.05pt]

\tikzstyle{labelsty}=[font=\scriptsize]

\tikzstyle{genericgraph}=[dashed,black!70!white]
\tikzstyle{indicatededge}=[pin={[pin distance=6pt,pin edge={thin,path fading=#1}]20:},pin={[pin distance=6pt,pin edge={thin,path fading=#1}]-20:},pin={[pin distance=10pt,pin edge={thin,path fading=#1}]2:}]

\tikzstyle{path}=[dotted,line width=1pt,-{Latex[width=2pt,length=3pt]}]

\crefname{enumi}{}{}
\Crefname{enumi}{Item}{Items}

\title{Constructing reflection-symmetric flexible realisations of graphs}
\author{Sean Dewar\thanks{School of Mathematics, University of Bristol, Fry Building,
Woodland Road,
Bristol,
BS8 1UG, UK, \texttt{sean.dewar@bristol.ac.uk}} \and
	Georg Grasegger\thanks{Johann Radon Institute for Computational and Applied Mathematics (RICAM),
	Austrian Academy of Sciences, Altenberger Straße 69, 4040 Linz, Austria, \texttt{georg.grasegger@ricam.oeaw.ac.at}} \and
	Jan Legersk\'y\thanks{Czech Technical University in Prague,
	Faculty of Information Technology, Thákurova 2700/9, 160 00 Prague 6, Czech Republic,
	\texttt{jan.legersky@fit.cvut.cz}, corresponding author}}
\date{}

\begin{document}
\maketitle
\begin{abstract}
We study reflection-symmetric realisations of symmetric graphs in the plane
that allow a continuous symmetry and edge-length preserving deformation.
To do so, we identify a necessary combinatorial condition on graphs with reflection-symmetric flexible realisations.
This condition is based on a specific type of edge colouring, where edges are assigned one of three colours in a symmetric way.
From some of these colourings we also construct concrete reflection-symmetric realisations with their corresponding symmetry preserving motion.
We study also a specific class of reflection-symmetric realisations consisting of triangles and parallelograms.

\vspace{20pt}

\vspace{5pt}
\end{abstract}

\section{Introduction}

A \emph{realisation} of a graph $G=(V_G,E_G)$ is the placement of the vertices in the plane, i.e.\ a map $p: V_G\rightarrow \mathbb R^2$, so that no two vertices joined by an edge are placed at the same point.
We define a realisation to be \emph{rigid} if any continuous edge-length preserving deformation corresponds to a rigid body motion.
However, graphs with rigid realisations may also have realisations that can be continuously deformed into non-congruent realisations while the edge lengths are preserved.
For some graphs these so-called \emph{flexible realisations} are difficult to find,
and their study has a long history.
For instance, Dixon \cite{Dixon} studied flexible realisations of the complete bipartite graph $K_{3,3}$.

Symmetric properties of graphs and realisation in relation to rigidity properties have been widely studied.
See for instance \cite{Bernstein,FowlerGuest,KaszanitzkySchulze,Schulze2010,SchulzeWhiteley},
or \cite{SchulzeCRC,SchulzeWhiteleyCRC} for an overview.
In this paper we consider reflection symmetry in particular.
Rigidity under this particular assumption has been previously investigated
in \cite{SchulzeTanigawa15,SchulzeWhiteley} and recently in~\cite{LaPortaSchulze}.
Flexibility and motions of symmetric graphs have also previously been
considered in \cite{KitsonSchulze,OwenSchulze,RossSchulzeWhiteley}.
However, each of them either use different symmetries, a different metric, or different classes of graphs than we do here.

Recently the existence of flexible realisations in the plane has been analysed
more deeply~\cite{flexibleLabelings,movableGraphs}.
The main idea therein is to relate the flexibility to a special type of edge
colouring by two colours, so-called \emph{NAC-colourings} (``No Almost Cycles'', see
\cite{flexibleLabelings}).
NAC-colourings characterise the existence of a flexible realisation in the plane
and give a construction of one such realisation with its flexible motion.
The \sage{} package \FlexRiLoG{}~\cite{flexrilog,FlexRiLoGPaper} offers functionality for
finding such colourings and constructing their corresponding motions.
The method of NAC-colourings has since been used for specific classes of graphs,
such as periodic \cite{Dperiodic2019} and infinite graphs, including Penrose patterns~\cite{DLinfinite}.
Also, the study of rigid and flexible frameworks consisting of parallelograms \cite{Bracing} and additional triangles \cite{TPframe} was recently developed using properties of NAC-colourings.
In \cite{SoCGMedia,RotFlex}, these methods were extended for constructing flexible realisations for rotation-symmetric graphs.

The aim of this paper is to analyse reflection-symmetric realisations of symmetric graphs.
In the spirit of NAC-colourings, we define a new type of edge colouring that allows us 
to give some necessary and sufficient conditions on graphs having a reflection-symmetric flex.
These new colourings are tailored for reflection symmetry but are based on NAC-colourings (which are recalled therefore in \Cref{sec:nac}).
In \Cref{sec:sym}, we recall notions from symmetry and define our edge-colouring.
We obtain a necessary condition for graphs having reflection-symmetric motions, which is found in \Cref{sec:necessary}.
We also construct reflection-symmetric flexes for some of the graphs satisfying the necessary condition in \Cref{sec:sufficient}.
\Cref{sec:ptp} extends the results for frameworks consisting of parallelograms and triangles to the reflection-symmetric case.
In \Cref{sec:conclusion}, we conclude that reflection-symmetry causes some crucial differences
to the general or rotation-symmetric case.

\section{Frameworks and NAC-colourings}\label{sec:nac}

A \emph{(planar) framework} is a pair $(G,p)$ consisting of a (finite and simple) graph $G=(V_G,E_G)$ and a map (known here as a \emph{realisation of $G$}) $p:V_G \rightarrow \mathbb{R}^2$, where $p(u) \neq p(v)$ if $uv$ is an edge of $G$.
Throughout the paper, we assume that all graphs are connected.

A \emph{flex} of a framework $(G,p)$ is any family of realisations $(p_t)_{t \in [0,1]}$ of $G$ where:
(i) for each vertex $v \in V_G$, the map $t \mapsto p_t(v)$ is continuous with $p_0(v) = p(v)$, and
(ii) for each edge $uv \in E_G$,
the map $t \mapsto \|p_t(u) - p_t(v)\|$ is constant.
A flex is said to be \emph{trivial} if it is the restriction of a rigid body motion to the points $\{p(v) : v \in V_G\}$.

A basic concept of recent analysis of flexible realisations are so called NAC-colourings.
They have been defined in \cite{flexibleLabelings} as a combinatorial property of graphs.
Since then they have been used to describe flexibility properties for various classes of 
graphs \cite{Dperiodic2019,RotFlex,DLinfinite,NAC-NP,movableGraphs,Bracing,TPframe}.

\begin{definition}
	Let $G$ be a graph.
	A \emph{NAC-colouring} of $G$ is a colouring\footnote{%
			Here an edge colouring is simply a map from the set of edges to a set of colours,
			not to be confused with the type of edge colouring used in the context of edge chromatic numbers.}
		of the edges $\delta\colon  E_G\rightarrow \{\text{\blue{}, \red{}}\}$
	that is surjective and for every cycle in $G$,
	either all edges have the same colour, or
	there are at least two edges of each colour (see \Cref{fig:nac}).
	
	\begin{figure}[ht]
		\centering
		\begin{tikzpicture}
			\begin{scope}
				\node[gvertex] (1) at (0,0) {};
				\node[gvertex] (2) at (1,0) {};
				\node[gvertex] (3) at (1,1) {};
				\node[gvertex] (4) at (0,1) {};
				\draw[edge,colR] (1)--(2) (3)--(4);
				\draw[edge,colB] (2)--(3) (4)--(1);
				\node[colGr] at (0.5,-0.5) {\faCheck};
			\end{scope}
			\begin{scope}[xshift=2cm]
				\node[gvertex] (1) at (0,0) {};
				\node[gvertex] (2) at (1,0) {};
				\node[gvertex] (3) at (1,1) {};
				\node[gvertex] (4) at (0,1) {};
				\draw[edge,colR] (1)--(2) (2)--(3);
				\draw[edge,colB] (3)--(4) (4)--(1);
				\node[colGr] at (0.5,-0.5) {\faCheck};
			\end{scope}
			\begin{scope}[xshift=5cm]
				\node[gvertex] (1) at (0,0) {};
				\node[gvertex] (2) at (1,0) {};
				\node[gvertex] (3) at (1,1) {};
				\node[gvertex] (4) at (0,1) {};
				\draw[edge,colR] (1)--(2) (2)--(3) (3)--(4) (4)--(1);
				\node[colR] at (0.5,-0.5) {\faTimes};
			\end{scope}
			\begin{scope}[xshift=7cm]
				\node[gvertex] (1) at (0,0) {};
				\node[gvertex] (2) at (1,0) {};
				\node[gvertex] (3) at (1,1) {};
				\node[gvertex] (4) at (0,1) {};
				\draw[edge,colR] (1)--(2) (2)--(3) (3)--(4);
				\draw[edge,colB] (4)--(1);
				\node[colR] at (0.5,-0.5) {\faTimes};
			\end{scope}
		\end{tikzpicture}
		\caption{Edge colourings of $C_4$ that are and are not NAC-colourings.}
		\label{fig:nac}
	\end{figure}
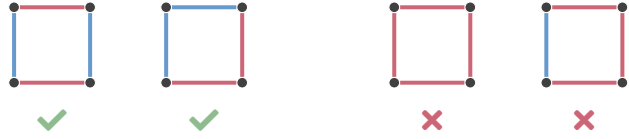
\end{definition}

These NAC-colourings combinatorially classify graphs
for which a flexible framework can be found:
a graph has a flexible framework if and only if it has a NAC-colouring~\cite{flexibleLabelings}.

\section{Reflection-symmetry in graphs, frameworks and edge colourings}\label{sec:sym}

An automorphism $\sigma$ of a graph $G$ is said to be a \emph{reflection} if it is distinct from the identity automorphism $I$ (in that $\sigma(v) \neq v$ for some vertex $v$) and $\sigma^2 = I$.
If a graph has a reflection~$\sigma$ then it is said to be \emph{reflection-symmetric (with reflection $\sigma$)}.
A vertex $v \in V_G$ of a reflection-symmetric graph is \emph{invariant} if $\sigma v = v$;
similarly, an edge $u v \in E_G$ is \emph{invariant} if either $\sigma u = u$ and $\sigma v = v$,
or if $\sigma u = v$ and $\sigma v = u$.

\begin{example}\label{ex:symg}
	The following figure shows two graphs with exactly three reflections apiece.
	
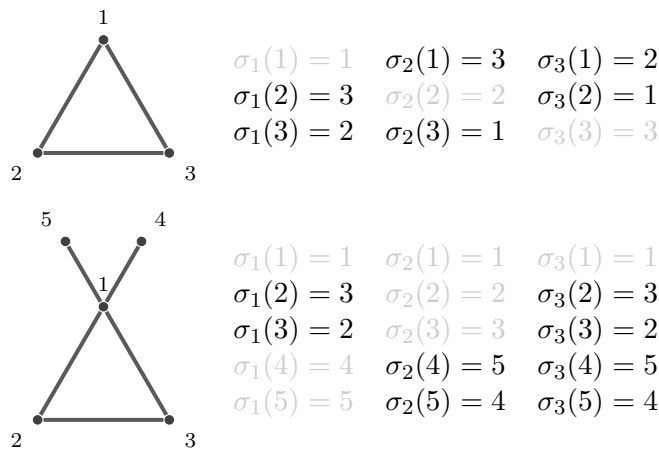
\begin{center}
	\begin{tikzpicture}
		\node[gvertex,label={[labelsty]90:1}] (a1) at (90:1) {};
		\node[gvertex,label={[labelsty]210:2}] (a2) at (210:1) {};
		\node[gvertex,label={[labelsty]-30:3}] (a3) at (-30:1) {};
		\draw[edge] (a1)--(a2) (a2)--(a3) (a3)--(a1);
		\coordinate (m) at ($(a2)!0.5!(a3)$);
		\node[align=left] at ($(2.5,0)+1/2*(a1)+1/2*(m)$) {\color{black!20!white}$\sigma_1(1)=1$\\$\sigma_1(2)=3$\\$\sigma_1(3)=2$};
		\node[align=left] at ($(4.5,0)+1/2*(a1)+1/2*(m)$) {$\sigma_2(1)=3$\\\color{black!20!white}$\sigma_2(2)=2$\\$\sigma_2(3)=1$};
		\node[align=left] at ($(6.5,0)+1/2*(a1)+1/2*(m)$) {$\sigma_3(1)=2$\\$\sigma_3(2)=1$\\\color{black!20!white}$\sigma_3(3)=3$};
	\end{tikzpicture}

	\begin{tikzpicture}
		\node[gvertex,label={[labelsty]90:1}] (a1) at (90:1) {};
		\node[gvertex,label={[labelsty]210:2}] (a2) at (210:1) {};
		\node[gvertex,label={[labelsty]-30:3}] (a3) at (-30:1) {};
		\node[gvertex,label={[labelsty]60:4}] (a4) at ($(a1)+(60:1)$) {};
		\node[gvertex,label={[labelsty]120:5}] (a5) at ($(a1)+(120:1)$) {};
		\draw[edge] (a1)--(a2) (a2)--(a3) (a3)--(a1) (a1)--(a4) (a1)--(a5);
		\coordinate (m1) at ($(a2)!0.5!(a3)$);
		\coordinate (m2) at ($(a4)!0.5!(a5)$);
		\node[align=left] at ($(2.5,0)+1/2*(m1)+1/2*(m2)$) {\color{black!20!white}$\sigma_1(1)=1$\\$\sigma_1(2)=3$\\$\sigma_1(3)=2$\\\color{black!20!white}$\sigma_1(4)=4$\\\color{black!20!white}$\sigma_1(5)=5$};
		\node[align=left] at ($(4.5,0)+1/2*(m1)+1/2*(m2)$) {\color{black!20!white}$\sigma_2(1)=1$\\\color{black!20!white}$\sigma_2(2)=2$\\\color{black!20!white}$\sigma_2(3)=3$\\$\sigma_2(4)=5$\\$\sigma_2(5)=4$};
		\node[align=left] at ($(6.5,0)+1/2*(m1)+1/2*(m2)$) {\color{black!20!white}$\sigma_3(1)=1$\\$\sigma_3(2)=3$\\$\sigma_3(3)=2$\\$\sigma_3(4)=5$\\$\sigma_3(5)=4$};
	\end{tikzpicture}
	\end{center}
\end{example}

\subsection{Reflection-symmetric frameworks}

A \emph{planar reflection} is a non-trivial isometry of the plane that has a line of fixed points.
W.l.o.g.\ we assume this reflection is through the $y$-axis.
This planar reflection is represented throughout by the matrix
\begin{align*}
	\tau := 
	\begin{bmatrix}
		-1 & 0 \\
		0 & 1
	\end{bmatrix}.
\end{align*}
With this,
we are now ready to give a formal definition for reflection-symmetric frameworks.
\begin{definition}
	Let $G$ be a reflection-symmetric graph with reflection $\sigma$.
	A realisation $p:V_G \rightarrow \mathbb{R}^2$ of $G$, resp.\ a framework $(G,p)$,
	is said to be \emph{reflection-symmetric (with reflection~$\sigma$)}
	if $p(\sigma v) = \tau p(v)$ for every vertex $v \in V_G$, where $\tau$ is the matrix defined above.
\end{definition}

\begin{example}
	Given the graph symmetries from \Cref{ex:symg} we can construct symmetric frameworks.
	Vertices that are fixed by $\sigma$ need to be placed on the symmetry line, which may cause degenerate subgraphs.
	However, we do not allow overlapping adjacent vertices.
	
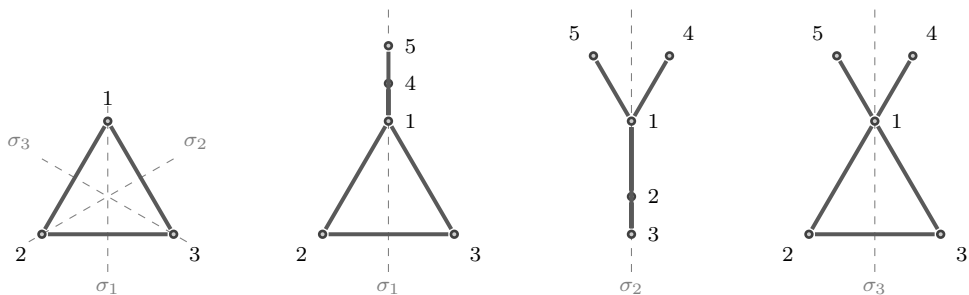
\begin{center}
	\begin{tikzpicture}
		\draw[sym] (-90:1)node[below,labelsty] {$\sigma_1$}--(90:1.2) (-30:1.2)--(150:1)node[above left,labelsty] {$\sigma_3$} (30:1)node[above right,labelsty] {$\sigma_2$}--(210:1.2);
		\node[fvertex,label={[labelsty]90:1}] (a1) at (90:1) {};
		\node[fvertex,label={[labelsty]210:2}] (a2) at (210:1) {};
		\node[fvertex,label={[labelsty]-30:3}] (a3) at (-30:1) {};
		\draw[edge] (a1)--(a2) (a2)--(a3) (a3)--(a1);
	\end{tikzpicture}
	\qquad
	\begin{tikzpicture}
		\draw[sym] (-90:1)node[below,labelsty] {$\sigma_1$} --(90:2.5);
		\node[fvertex,label={[labelsty]0:1}] (a1) at (90:1) {};
		\node[fvertex,label={[labelsty]210:2}] (a2) at (210:1) {};
		\node[fvertex,label={[labelsty]-30:3}] (a3) at (-30:1) {};
		\node[fvertex,label={[labelsty]0:4}] (a4) at ($(a1)+(90:0.5)$) {};
		\node[fvertex,label={[labelsty]0:5}] (a5) at ($(a1)+(90:1)$) {};
		\draw[edge] (a1)--(a2) (a2)--(a3) (a3)--(a1) (a1)--(a4) (a1)--(a5);
	\end{tikzpicture}
	\qquad
	\begin{tikzpicture}
		\draw[sym] (-90:1) node[below,labelsty] {$\sigma_2$}--(90:2.5);
		\node[fvertex,label={[labelsty]0:1}] (a1) at (90:1) {};
		\node[fvertex,label={[labelsty]0:2}] (a2) at (-90:0) {};
		\node[fvertex,label={[labelsty]0:3}] (a3) at (-90:0.5) {};
		\node[fvertex,label={[labelsty]60:4}] (a4) at ($(a1)+(60:1)$) {};
		\node[fvertex,label={[labelsty]120:5}] (a5) at ($(a1)+(120:1)$) {};
		\draw[edge] (a1)--(a2) (a2)--(a3) (a3)--(a1) (a1)--(a4) (a1)--(a5);
	\end{tikzpicture}
	\qquad
	\begin{tikzpicture}
		\draw[sym] (-90:1) node[below,labelsty] {$\sigma_3$}--(90:2.5);
		\node[fvertex,label={[labelsty]0:1}] (a1) at (90:1) {};
		\node[fvertex,label={[labelsty]210:2}] (a2) at (210:1) {};
		\node[fvertex,label={[labelsty]-30:3}] (a3) at (-30:1) {};
		\node[fvertex,label={[labelsty]60:4}] (a4) at ($(a1)+(60:1)$) {};
		\node[fvertex,label={[labelsty]120:5}] (a5) at ($(a1)+(120:1)$) {};
		\draw[edge] (a1)--(a2) (a2)--(a3) (a3)--(a1) (a1)--(a4) (a1)--(a5);
	\end{tikzpicture}
	\end{center}
\end{example}

Let $(p_t)_{t \in [0,1]}$ be a flex of a reflection-symmetric framework $(G,p)$ with reflection $\sigma$.
We say $p_t$ is a \emph{reflection-symmetric flex} if for each $t \in [0,1]$, the framework $(G,p_t)$ is reflection-symmetric with reflection $\sigma$.

\begin{definition}
	A reflection-symmetric framework is \emph{reflection-symmetric rigid} if every reflection-symmetric flex of itself is trivial;
	otherwise, the framework is said to be \emph{reflection-symmetric flexible}.
\end{definition}

\subsection{RS-colourings}

We begin with defining a new type of edge colouring based on the idea
of NAC-colourings but tailored to the needs of the reflectional setting.

\begin{definition}
	\label{def:pseudoRScol}
	Let $G$ be a reflection-symmetric graph with reflection $\sigma$. 
	An edge colouring
	$\delta:E_G\rightarrow \{\red,\blue,\gold\}$ of $G$
	is a \emph{pseudo-RS-colouring} if the following holds:
	\begin{enumerate}
		\item\label{it:containsRedBlue} $ \{ \red{},\blue{} \}\subseteq \delta(E_G) \subseteq \{ \red{},\blue{}, \gold{} \}$,
		\item\label{it:gold2blue} changing \gold{} to \blue{} results in a NAC-colouring,
		\item\label{it:gold2red} changing \gold{} to \red{} results in a NAC-colouring,
		\item\label{it:symmetryRedBlue} $\delta(e) = \red{}$ if and only
			if $\delta(\sigma e) = \blue{}$ for all $e\in E_G$.
	\end{enumerate}
	A cycle in $G$ is \emph{almost red-blue} if it has exactly one gold edge. 
	Pseudo-RS-colourings $\delta,\conj{\delta}$ are called \emph{conjugated} if
	$\conj{\delta}=\delta \circ \sigma$.
\end{definition}

\Cref{fig:pseudoRScol} shows pseudo-RS-colourings (RS stands for reflection-symmetric)
of a reflection-symmetric graph.
 
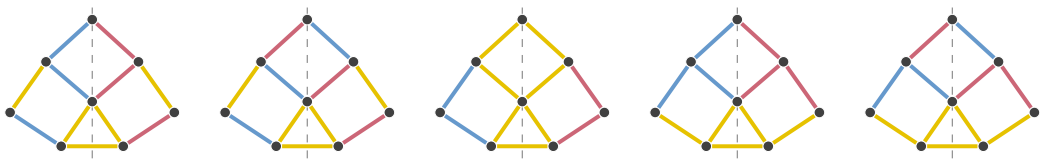
\begin{figure}[ht]
  \begin{center}
    \begin{tikzpicture}[scale=0.8]
			\draw[sym] (0,-0.2) -- (0,2.3);
			\node[gvertex] (1) at (-0.51, 0.) {};
			\node[gvertex] (2) at (0.51, 0.) {};
			\node[gvertex] (3) at (0., 0.74) {};
			\node[gvertex] (4) at (-1.35, 0.56) {};
			\node[gvertex] (5) at (-0.76, 1.4) {};
			\node[gvertex] (6) at (0., 2.1) {};
			\node[gvertex] (7) at (0.76, 1.4) {};
			\node[gvertex] (8) at (1.35, 0.56) {};
			\draw[gedge] (1)edge(2);
			\draw[gedge] (1)edge(3);
			\draw[bedge] (1)edge(4);
			\draw[gedge] (2)edge(3);
			\draw[redge] (2)edge(8);
			\draw[bedge] (3)edge(5);
			\draw[redge] (3)edge(7);
			\draw[gedge] (4)edge(5);
			\draw[bedge] (5)edge(6);
			\draw[redge] (6)edge(7);
			\draw[gedge] (7)edge(8);
		\end{tikzpicture}
		\quad
		\begin{tikzpicture}[scale=0.8]
			\draw[sym] (0,-0.2) -- (0,2.3);
			\node[gvertex] (1) at (-0.51, 0.) {};
			\node[gvertex] (2) at (0.51, 0.) {};
			\node[gvertex] (3) at (0., 0.74) {};
			\node[gvertex] (4) at (-1.35, 0.56) {};
			\node[gvertex] (5) at (-0.76, 1.4) {};
			\node[gvertex] (6) at (0., 2.1) {};
			\node[gvertex] (7) at (0.76, 1.4) {};
			\node[gvertex] (8) at (1.35, 0.56) {};
			\draw[gedge] (1)edge(2);
			\draw[gedge] (1)edge(3);
			\draw[bedge] (1)edge(4);
			\draw[gedge] (2)edge(3);
			\draw[redge] (2)edge(8);
			\draw[bedge] (3)edge(5);
			\draw[redge] (3)edge(7);
			\draw[gedge] (4)edge(5);
			\draw[redge] (5)edge(6);
			\draw[bedge] (6)edge(7);
			\draw[gedge] (7)edge(8);
		\end{tikzpicture}
		\quad
		\begin{tikzpicture}[scale=0.8]
			\draw[sym] (0,-0.2) -- (0,2.3);
			\node[gvertex] (1) at (-0.51, 0.) {};
			\node[gvertex] (2) at (0.51, 0.) {};
			\node[gvertex] (3) at (0., 0.74) {};
			\node[gvertex] (4) at (-1.35, 0.56) {};
			\node[gvertex] (5) at (-0.76, 1.4) {};
			\node[gvertex] (6) at (0., 2.1) {};
			\node[gvertex] (7) at (0.76, 1.4) {};
			\node[gvertex] (8) at (1.35, 0.56) {};
			\draw[gedge] (1)edge(2);
			\draw[gedge] (1)edge(3);
			\draw[bedge] (1)edge(4);
			\draw[gedge] (2)edge(3);
			\draw[redge] (2)edge(8);
			\draw[gedge] (3)edge(5);
			\draw[gedge] (3)edge(7);
			\draw[bedge] (4)edge(5);
			\draw[gedge] (5)edge(6);
			\draw[gedge] (6)edge(7);
			\draw[redge] (7)edge(8);
		\end{tikzpicture}
		\quad
		\begin{tikzpicture}[scale=0.8]
			\draw[sym] (0,-0.2) -- (0,2.3);
			\node[gvertex] (1) at (-0.51, 0.) {};
			\node[gvertex] (2) at (0.51, 0.) {};
			\node[gvertex] (3) at (0., 0.74) {};
			\node[gvertex] (4) at (-1.35, 0.56) {};
			\node[gvertex] (5) at (-0.76, 1.4) {};
			\node[gvertex] (6) at (0., 2.1) {};
			\node[gvertex] (7) at (0.76, 1.4) {};
			\node[gvertex] (8) at (1.35, 0.56) {};
			\draw[gedge] (1)edge(2);
			\draw[gedge] (1)edge(3);
			\draw[gedge] (1)edge(4);
			\draw[gedge] (2)edge(3);
			\draw[gedge] (2)edge(8);
			\draw[bedge] (3)edge(5);
			\draw[redge] (3)edge(7);
			\draw[bedge] (4)edge(5);
			\draw[bedge] (5)edge(6);
			\draw[redge] (6)edge(7);
			\draw[redge] (7)edge(8);
		\end{tikzpicture}
		\quad
		\begin{tikzpicture}[scale=0.8]
			\draw[sym] (0,-0.2) -- (0,2.3);
			\node[gvertex] (1) at (-0.51, 0.) {};
			\node[gvertex] (2) at (0.51, 0.) {};
			\node[gvertex] (3) at (0., 0.74) {};
			\node[gvertex] (4) at (-1.35, 0.56) {};
			\node[gvertex] (5) at (-0.76, 1.4) {};
			\node[gvertex] (6) at (0., 2.1) {};
			\node[gvertex] (7) at (0.76, 1.4) {};
			\node[gvertex] (8) at (1.35, 0.56) {};
			\draw[gedge] (1)edge(2);
			\draw[gedge] (1)edge(3);
			\draw[gedge] (1)edge(4);
			\draw[gedge] (2)edge(3);
			\draw[gedge] (2)edge(8);
			\draw[bedge] (3)edge(5);
			\draw[redge] (3)edge(7);
			\draw[bedge] (4)edge(5);
			\draw[redge] (5)edge(6);
			\draw[bedge] (6)edge(7);
			\draw[redge] (7)edge(8);
		\end{tikzpicture}
  \end{center}
  \caption{All pseudo-RS-colourings of a reflection-symmetric graph up to conjugates.}
  \label{fig:pseudoRScol}
\end{figure}

\begin{remark}
	For a pseudo-RS-colouring of graph $G$, it follows from \cref{it:symmetryRedBlue} that
	\cref{it:gold2blue} holds if and only if \cref{it:gold2red} holds.
	Furthermore, if $\delta(e) = \gold{}$ then $\delta(\sigma e) = \gold{}$ for all $e\in E_G$.
\end{remark}
Actually, we need a specific type of pseudo-RS-colourings.
\begin{definition}
	A pseudo-RS-colouring $\delta$ of graph $G$ is called an \emph{RS-colouring} if:
	\begin{enumerate}
		\item $(G,\delta)$ has no almost red-blue cycle, or
		\item for every almost red-blue cycle,
		there exist edges $e_1,e_2$ in the cycle and another pseudo-RS-colouring $\delta'$ of $G$ (called a \emph{certificate})
		such that $\delta(e_1)=\delta(e_2)$ and $\delta'(e_1)\neq\delta'(e_2)$.
	\end{enumerate}
\end{definition}
All pseudo-RS-colourings from \Cref{fig:pseudoRScol} contain no almost red-blue cycles and are therefore RS-colourings.
\Cref{fig:RScolourings} shows an example of an RS-colouring with an almost red-blue cycle
and the corresponding certificate pseudo-RS-colouring.
Not all pseudo-RS-colourings are RS-colourings, however; see \Cref{fig:nonRScolourings} for two such examples.

\begin{figure}[ht]
	\centering
	\begin{tikzpicture}
		\node[gvertex] (1) at (0.6, -0.8) {};
		\node[gvertex] (2) at (0.9, 0.3) {};
		\node[gvertex] (3) at (0., 1.8) {};
		\node[gvertex] (4) at (-0.9, 0.3) {};
		\node[gvertex] (5) at (-0.6, -0.8) {};
		
		\node[gvertex] (6) at (1.4, -0.5) {};
		\node[gvertex] (7) at (-1.4, -0.5) {};
		\node[gvertex] (8) at (0, -1.8) {};
		
		\node[gvertex] (1a) at (1.6, -1.5) {};
		\node[gvertex] (2a) at (1.9, -0.4) {};
		\node[gvertex] (6a) at (2.4, -1.2) {};

		\node[gvertex] (5a) at (-1.6, -1.5) {};
		\node[gvertex] (4a) at (-1.9, -0.4) {};
		\node[gvertex] (7a) at (-2.4, -1.2) {};
		
		\node[gvertex] (8b) at (1, -2.5) {};
		\node[gvertex] (5b) at (0.4, -1.5) {};
		\node[gvertex] (1b) at (1a) {};
		
		\node[gvertex] (8c) at (-1, -2.5) {};
		\node[gvertex] (1c) at (-0.4, -1.5) {};
		\node[gvertex] (5c) at (5a) {};

		\draw[gedge] (1)edge(5) (8)edge(5) (1)edge(8) 
			(1b)edge(5b) (8b)edge(5b) (1b)edge(8b) 
			(1c)edge(5c) (8c)edge(5c) (1c)edge(8c);
		\draw[bedge] (4)edge(3) (4a)edge(3)
			(2)edge(1) (6)edge(1) (2)edge(6)
			(2a)edge(1a) (6a)edge(1a) (2a)edge(6a)
			(4)edge(4a) (5)edge(5a) (7)edge(7a)
			(8)edge(8c) (1)edge(1c);
		\draw[redge]  (2)edge(3) (2a)edge(3)
			(5)edge(4) (7)edge(5) (4)edge(7)
			(1)edge(1a) (2)edge(2a) (6)edge(6a) 
			(5a)edge(4a) (7a)edge(5a) (4a)edge(7a)
			(8)edge(8b) (5)edge(5b);	
	\begin{scope}[xshift=6cm]
		\node[gvertex] (1) at (0.6, -0.8) {};
		\node[gvertex] (2) at (0.9, 0.3) {};
		\node[gvertex] (3) at (0., 1.8) {};
		\node[gvertex] (4) at (-0.9, 0.3) {};
		\node[gvertex] (5) at (-0.6, -0.8) {};
		
		\node[gvertex] (6) at (1.4, -0.5) {};
		\node[gvertex] (7) at (-1.4, -0.5) {};
		\node[gvertex] (8) at (0, -1.8) {};
		
		\node[gvertex] (1a) at (1.6, -1.5) {};
		\node[gvertex] (2a) at (1.9, -0.4) {};
		\node[gvertex] (6a) at (2.4, -1.2) {};

		\node[gvertex] (5a) at (-1.6, -1.5) {};
		\node[gvertex] (4a) at (-1.9, -0.4) {};
		\node[gvertex] (7a) at (-2.4, -1.2) {};
		
		\node[gvertex] (8b) at (1, -2.5) {};
		\node[gvertex] (5b) at (0.4, -1.5) {};
		\node[gvertex] (1b) at (1a) {};
		
		\node[gvertex] (8c) at (-1, -2.5) {};
		\node[gvertex] (1c) at (-0.4, -1.5) {};
		\node[gvertex] (5c) at (5a) {};
		
		\draw[gedge] (1)edge(5) (8)edge(5) (1)edge(8) 
			(1b)edge(5b) (8b)edge(5b) (1b)edge(8b) 
			(1c)edge(5c) (8c)edge(5c) (1c)edge(8c);
		\draw[bedge] 
			(2)edge(1) (6)edge(1) (2)edge(6)
			(2a)edge(1a) (6a)edge(1a) (2a)edge(6a)
			
			(2)edge(3) (2a)edge(3) (2)edge(2a) (1)edge(1a) (8)edge(8b) (5)edge(5b)(6)edge(6a) 
			;
		\draw[redge]  
			(5)edge(4) (7)edge(5) (4)edge(7)
			(5a)edge(4a) (7a)edge(5a) (4a)edge(7a)
			
			(4)edge(4a) (4)edge(3) (4a)edge(3)
			(8)edge(8c) (5)edge(5a) (7)edge(7a) (1)edge(1c)
			;
	\end{scope}
	\end{tikzpicture}
	\caption{Two RS-colourings of the same graph which are certificates for each other.}
	\label{fig:RScolourings}
\end{figure}
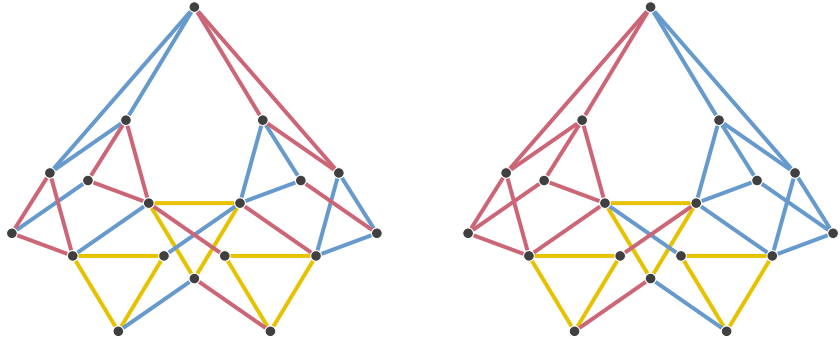

\begin{figure}[ht]
	\centering
	\begin{tikzpicture}
	  \node[gvertex] (1) at (0.587785, -0.809017) {};
		\node[gvertex] (2) at (0.951057, 0.309017) {};
		\node[gvertex] (3) at (0., 1.) {};
		\node[gvertex] (4) at (-0.951057, 0.309017) {};
		\node[gvertex] (5) at (-0.587785, -0.809017) {};
		
		\node[gvertex] (6) at (0.2, 0.125) {};
		\node[gvertex] (7) at (-0.2, 0.125) {};
		\draw[gedge] (1)edge(5);
		\draw[bedge] (2)edge(3) (2)edge(1) (1)edge(6) (2)edge(6) (3)edge(6);
		\draw[redge] (4)edge(3) (5)edge(4) (4)edge(7) (5)edge(7) (3)edge(7);
	\end{tikzpicture}
	\qquad
	\begin{tikzpicture}
		  	\node[gvertex] (a) at (0.587785, -0.809017) {};
			\node[gvertex] (b) at (-0.587785, -0.809017) {};
			\draw[gedge] (a)edge(b);
		\begin{scope}[rotate around={60:(b)}]
		  	\node[gvertex] (1) at (0.587785, -0.809017) {};
			\node[gvertex] (2) at (0.951057, 0.309017) {};
			\node[gvertex] (3) at (0., 1.) {};
			\node[gvertex] (4) at (-0.951057, 0.309017) {};
			\node[gvertex] (5) at (-0.587785, -0.809017) {};
			
			\node[gvertex] (6) at (0.2, 0.125) {};
			\node[gvertex] (7) at (-0.2, 0.125) {};
			\draw[gedge] (1)edge(5);
			\draw[bedge] (2)edge(3) (2)edge(1) (1)edge(6) (2)edge(6) (3)edge(6);
			\draw[redge] (4)edge(3) (5)edge(4) (4)edge(7) (5)edge(7) (3)edge(7);
		\end{scope}
		\begin{scope}[rotate around={-60:(a)}]
		  	\node[gvertex] (1) at (0.587785, -0.809017) {};
			\node[gvertex] (2) at (0.951057, 0.309017) {};
			\node[gvertex] (3) at (0., 1.) {};
			\node[gvertex] (4) at (-0.951057, 0.309017) {};
			\node[gvertex] (5) at (-0.587785, -0.809017) {};
			
			\node[gvertex] (6) at (0.2, 0.125) {};
			\node[gvertex] (7) at (-0.2, 0.125) {};
			\draw[gedge] (1)edge(5);
			\draw[bedge] (2)edge(3) (2)edge(1) (1)edge(6) (2)edge(6) (3)edge(6);
			\draw[redge] (4)edge(3) (5)edge(4) (4)edge(7) (5)edge(7) (3)edge(7);
		\end{scope}
	\end{tikzpicture}
	\caption{Examples of pseudo-RS-colourings that are not RS-colourings.
	Every reflection-symmetric realisation of the left graph is reflection-symmetric rigid.
	The right graph, however, admits both an RS-colouring and a reflection-symmetric flexible realisation.}
	\label{fig:nonRScolourings}
\end{figure}
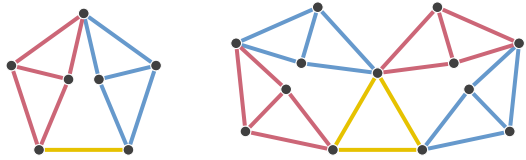

Finding a NAC-colouring is NP-complete \cite{NAC-NP},
which is true also for pseudo-RS-colourings.

\begin{proposition}
	It is NP-complete to decide whether a reflection-symmetric graph has a pseudo-RS-colouring
	and it is NP-hard to decide whether a reflection-symmetric graph has an RS-colouring.
\end{proposition}
\begin{proof}
	We proceed by reduction from the question whether a graph $G$ has a NAC-colouring.
	Let $H$ be obtained by gluing two copies of $G$ on an edge $f$
	so that $H$ is reflection-symmetric with a symmetry $\sigma$ that keeps only the $f$ fixed.
	Let us assume we have a NAC-colouring $\delta$ of $G$.
	Then we obtain a pseudo-RS-colouring $\delta'$ of $H$ in the following way.
	Let us assume that $\delta(f)=\red$ and
	denote the two copies of $G$ in $H$ by $G$ and~$\sigma G$.
	We now set $\delta'(f)=\gold$, and for each $e\in E_G$ and $\sigma e\in E_{\sigma G}$ we set
	(see also \Cref{fig:np})
	\begin{align*}
		\delta'(e)&=
		\begin{cases}
			\gold & \text{if } \delta(e)=\red,\\
			\blue & \text{if } \delta(e)=\blue,
		\end{cases}
		&
		\delta'(\sigma e)&=
		\begin{cases}
			\gold & \text{if } \delta(e)=\red,\\
			\red & \text{if } \delta(e)=\blue.
		\end{cases}
	\end{align*}
	In other words: in $G$ we switch red to gold and in $\sigma G$ we switch red to gold and blue to red.

	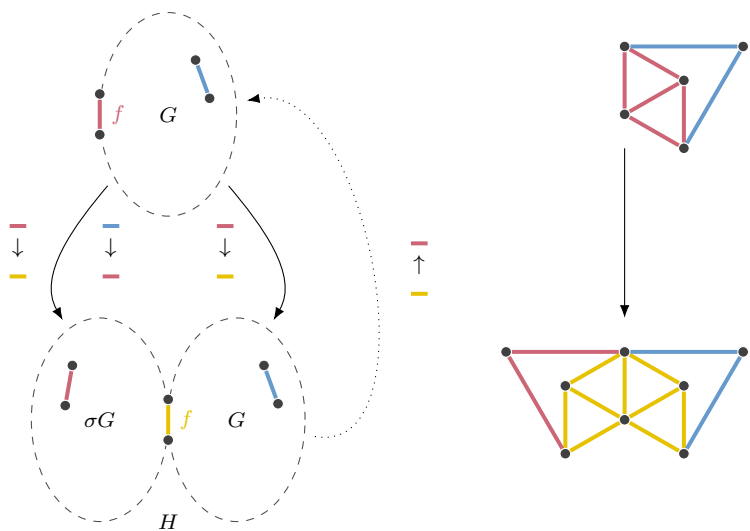
\begin{figure}[ht]
		\centering
		\begin{tikzpicture}[scale=0.9]
			\clip (-2.7,-6.5) rectangle (4,1.9);
			\begin{scope}
				\draw[genericgraph,name path=G] (0,0) circle[x radius=1cm,y radius=1.5cm];
				\path[name path=La] (0,0)--(-50:5);
				\path[name intersections={of=G and La,by={pa}}];
				\path[name path=Lb] (0,0)--(-130:5);
				\path[name intersections={of=G and Lb,by={pb}}];
				\path[name path=Lc] (0,0)--(10:5);
				\path[name intersections={of=G and Lc,by={pc}}];
				\node[labelsty] at (0,0) {$G$};
				\node[gvertex,indicatededge=east,rotate=0] (1) at (-1,0.3) {};
				\node[gvertex,indicatededge=east,rotate=0] (2) at (-1,-0.3) {};
				\draw[edge,colR] (1)to node[right,labelsty] {$f$} (2);
				\node[gvertex,indicatededge=west,rotate=180] (3) at (0.4,0.8) {};
				\node[gvertex,indicatededge=south,rotate=-90] (4) at ($(3)+(-70:0.6)$) {};
				\draw[edge,colB] (3)--(4);
			\end{scope}
			\begin{scope}[yshift=-4.5cm]
				\draw[genericgraph,name path=G1] (1,0) circle[x radius=1cm,y radius=1.5cm];
				\draw[genericgraph,name path=G2] (-1,0) circle[x radius=1cm,y radius=1.5cm];
				\path[name path=La2] (1,0)--(60:5);
				\path[name intersections={of=G1 and La2,by={pa2}}];
				\path[name path=Lb2] (-1,0)--(120:5);
				\path[name intersections={of=G2 and Lb2,by={pb2}}];
				\path[name path=Lc2] (1,0)--(-10:5);
				\path[name intersections={of=G1 and Lc2,by={pc2}}];
				\node[labelsty] at (1,0) {$G$};
				\node[labelsty] at (-1,0) {$\sigma G$};
				\node[labelsty] at (0,-1.5) {$H$};
				\node[gvertex,indicatededge=east,rotate=0] (1) at (0,0.3) {};
				\node[gvertex,indicatededge=east,rotate=0] (2) at (0,-0.3) {};
				\node[gvertex,indicatededge=west,rotate=180] at (1) {};
				\node[gvertex,indicatededge=west,rotate=180] at (2) {};
				\draw[edge,colG] (1)to node[right,labelsty] {$f$}(2);
				\node[gvertex,indicatededge=west,rotate=180] (3) at (1.4,0.8) {};
				\node[gvertex,indicatededge=south,rotate=-90] (4) at ($(3)+(-70:0.6)$) {};
				\node[gvertex,indicatededge=east,rotate=0] (3s) at (-1.4,0.8) {};
				\node[gvertex,indicatededge=south,rotate=-90] (4s) at ($(3s)+(-100:0.6)$) {};
				\draw[edge,colB] (3)--(4);
				\draw[edge,colR] (3s)--(4s);
			\end{scope}
			\begin{scope}[trans/.style={-Latex[],shorten <= 4pt,shorten >= 4pt}]
				\draw[trans] (pa) to[out=-50,in=60] node[left=0.2cm,labelsty,align=center] {\tikz{\draw[edge,colR,-] (0,0)--(0.5,0);}\\$\downarrow$\\\tikz{\draw[edge,colG,-] (0,0)--(0.5,0);}} (pa2);
				\draw[trans] (pb) to[out=-130,in=120] node[left=0.2cm,labelsty,align=center] {\tikz{\draw[edge,colR,-] (0,0)--(0.5,0);}\\$\downarrow$\\\tikz{\draw[edge,colG,-] (0,0)--(0.5,0);}} node[right=0.2cm,labelsty,align=center] {\tikz{\draw[edge,colB,-] (0,0)--(0.5,0);}\\$\downarrow$\\\tikz{\draw[edge,colR,-] (0,0)--(0.5,0);}} (pb2);
				\draw[trans,dotted] (pc2) to[out=-10,in=10] node[right=0.25cm,labelsty,align=center] {\tikz{\draw[edge,solid,colR,-] (0,0)--(0.5,0);}\\$\uparrow$\\\tikz{\draw[edge,solid,colG,-] (0,0)--(0.5,0);}} (pc);
			\end{scope}
		\end{tikzpicture}
		\quad
		\begin{tikzpicture}[scale=0.9]
			\clip (-2.1,-6.5) rectangle (2,1.9);
			\begin{scope}
				\node[gvertex] (a) at (90:1) {};
				\node[gvertex] (b) at (-30:1) {};
				\node[gvertex] (c) at (30:2) {};
				\node[gvertex] (e) at (30:1) {};
				\node[gvertex] (f) at (0,0) {};
				\draw[edge,colR] (a)edge(e) (a)edge(f) (b)edge(e) (b)edge(f) (e)edge(f);
				\draw[edge,colB] (a)edge(c) (b)edge(c);
			\end{scope}
			\begin{scope}[yshift=-4.5cm]
				\node[gvertex] (a) at (90:1) {};
				\node[gvertex] (b) at (-30:1) {};
				\node[gvertex] (c) at (30:2) {};
				\node[gvertex] (e) at (30:1) {};
				\node[gvertex] (f) at (0,0) {};
				\draw[edge,colG] (a)edge(e) (a)edge(f) (b)edge(e) (b)edge(f) (e)edge(f);
				\draw[edge,colB] (a)edge(c) (b)edge(c);

				\node[gvertex] (b2) at (210:1) {};
				\node[gvertex] (c2) at (150:2) {};
				\node[gvertex] (e2) at (150:1) {};

				\draw[edge,colG] (a)edge(e2) (b2)edge(e2) (b2)edge(f) (e2)edge(f);
				\draw[edge,colR] (a)edge(c2) (b2)edge(c2);
			\end{scope}
			\begin{scope}[trans/.style={-Latex[]}]
				\draw[trans] (0,-0.5)--(0,-3);
			\end{scope}
		\end{tikzpicture}

		\caption{The conversion of a NAC-colouring of $G$ to a pseudo-RS-colouring of $H$ and an example.}
		\label{fig:np}
	\end{figure}

	Since $\delta$ is a NAC-colouring of $G$, it contains both red and blue edges. If $\delta(e)=\blue$ then $\delta'(e)=\blue$ and $\delta'(\sigma e)=\red$, hence $\delta'$ contains all three colours (condition \cref{it:containsRedBlue}).
	When we change all gold edges in $\delta'$ to blue,
	then $G$ is monochromatic in blue and
	$\sigma G$ is coloured as $\delta$ with red and blue interchanged,
	hence it is a NAC-colouring (condition~\cref{it:gold2blue}).
	When we change all gold edges in $\delta'$ to red,
	then $\sigma G$ is monochromatic in red and
	$G$ is coloured the same as in~$\delta$,
	hence it is a NAC-colouring (condition \cref{it:gold2red}).
	The symmetry condition \cref{it:symmetryRedBlue} in $\delta'$ follows from the definition.
	Hence, $\delta'$ is a pseudo-RS-colouring.
	
	In fact, $\delta'$ is an RS-colouring also.
	Suppose that there is an almost red-blue cycle.
	W.l.o.g.\ we assume its gold edge is in $G$ and it contains a blue edge.
	If the gold edge is $f$, then the cycle does not contain any edge of $\sigma G$.
	Thus, it contains no red edges, contradicting $\delta'$ is a pseudo-RS-colouring.
	If the gold edge is not $f$, then it must contain also an edge of $\sigma G$,
	since it requires at least one red edge.
	Then the cycle passes through both vertices of $f$.
	However, we can now create a red-gold cycle with exactly one gold edge
	by removing the blue edges and adding $f$, which is a contradiction.
	Hence, $\delta'$ is an RS-colouring.

	Now assume we have any (pseudo-)RS-colouring of $H$.
	When we replace the gold by red we get a NAC-colouring of $H$ by condition \cref{it:gold2red}.
	This is also a NAC-colouring on any subgraph or it is monochromatic on that subgraph.
	Hence, either $G$ or $\sigma G\cong G$ has a NAC-colouring.
	Since finding a NAC-colouring is NP-complete \cite[Theorem 3.5]{NAC-NP},
	we know now that finding a (pseudo-)RS-colouring is NP-hard.
	Since checking whether a given colouring is a NAC-colouring
	can be done in polynomial time, so it is for a pseudo-RS-colouring.
\end{proof}
We remark that it is unknown if the existence of an RS-colouring is NP-complete,
since we do not know if the number of almost red-blue cycles
is bounded by a polynomial in the number of vertices.

\section{Reflection-symmetric flex implies RS-colouring}
\label{sec:necessary}
In this section we determine a necessary condition for the existence of a reflection-symmetric realisation.
In particular, we show that a reflection-symmetric flex induces an RS-colouring.

We begin by recalling a simple variable transformation for the algebraic equation representing the edges.

\begin{definition}
	\label{def:RSalgebraicMotion}
	Let $(G,p)$ be a reflection-symmetric framework and $v_0\in V_G$.
	The \emph{induced edge lengths} of $(G,p)$ is
	the map $\lambda_p:E_G \rightarrow \RR^+$ given by $\lambda_p(uv)=||p(u)-p(v)||$.
	We define $\RSconfigSpace(G,p)$ to be the algebraic subset of $\CC^{2|V_G|}$ defined by the equations
	\begin{align}
		(x_u - x_v)^2 + (y_u - y_v)^2 &= \lambda_p(uv)^2 \text{ for all } uv \in E_G, \label{eq:edgeLengths} \\
		y_{v_0} = 0, \qquad	x_{\sigma v} = - x_{v}, \qquad y_{\sigma v} &= y_v \text{ for all } v \in V_G\,. \label{eq:symmetry}
	\end{align}
	An irreducible algebraic curve $\motion\subseteq\RSconfigSpace(G,p)$ is called
	a \emph{reflection-symmetric algebraic motion} of $(G,p)$.
	For $u,v\in V_G$, we define the following functions in the complex function field $\CC(\motion)$:
	\begin{equation*}
		\Wfun{u}{v} = (x_u - x_v) + \ci(y_u - y_v) \qquad\text{ and }\qquad
		\Zfun{u}{v} = (x_u - x_v) - \ci(y_u - y_v)
	\end{equation*} 
\end{definition}

The main concept in this section are valuations of the function field of an algebraic motion.
As we see later, thresholds on these valuations give us the three colours.
\begin{definition}
	Let $\motion$ be an irreducible algebraic curve.
	A \emph{valuation} of the complex function field $\CC(\motion)$
	is a map $\nu:\CC(\motion)\setminus\{0\} \rightarrow \QQ$ such that
	\begin{enumerate}
	  \item $\nu(ab)=\nu(a)+\nu(b)$ for all $a,b\in \CC(\motion)\setminus\{0\}$,
	  \item $\nu(a+b)\geq \min\{\nu(a),\nu(b)\}$ for all $a,b\in \CC(\motion)\setminus\{0\}$ such that $a+b\neq 0$, and
	  \item $\nu(\CC \setminus \{0\})=\{0\}$.
	\end{enumerate}
\end{definition}

\begin{remark}
	\label{rem:WZ}
	For any $uv \in E_G$ and a valuation $\nu$, we have from \Cref{eq:edgeLengths,eq:symmetry}: 
	\begin{align*}
		0 = \nu(\lambda_p(uv)) = \nu(\Wfun{u}{v}\Zfun{u}{v}) = \nu(\Wfun{u}{v}) + \nu(\Zfun{u}{v})\qquad \text{and}\qquad
		\nu(\Wfun{\sigma u}{\sigma v}) = \nu(\Zfun{u}{v}).
	\end{align*}
\end{remark}

We observe a relation between valuations and angle changes during a flex.
\begin{lemma}\label{lem:nonConstAngle2motionWithValuation}
	Let $(G,p)$ be a reflection-symmetric flexible framework with edges $u_1v_1$ and $u_2v_2$. 
	The following holds:
	\begin{enumerate}
	  \item\label{it:flex2motion} If the angle between edges $u_1v_1$ and $u_2v_2$ changes along some flex of $(G,p)$, 
	  		then there exists a reflection-symmetric algebraic motion $\motion \subset \mathcal{V}_s(G,p)$
	  		such that the angle between edges $u_1v_1$ and $u_2v_2$ changes along the motion.
	  \item If $\motion \subset \mathcal{V}_s(G,p)$ is a reflection-symmetric algebraic motion 
	  		such that the angle between edges $u_1v_1$ and $u_2v_2$ changes along the motion,
	  		then there exists a valuation~$\nu$ of $\CC(\motion)$
			such that $\nu(\Wfun{u_1}{v_1})<\nu(\Wfun{u_2}{v_2})$.
	\end{enumerate}
\end{lemma}
\begin{proof}
	This is shown by following the proof of \cite[Lemma~2.7]{DLinfinite} with restricting
	set corresponding to the reflection symmetry --- namely, \Cref{eq:symmetry}.
\end{proof}

Now we show how to construct a pseudo-RS-colouring from a valuation.
\begin{lemma}
	\label{lem:valuationToPseudoRS}
	Let $\alpha > 0$, $\motion$ be a reflection-symmetric algebraic motion of $(G,p)$,
	and $\nu$ be a valuation of $\CC(\motion)$ such that $\nu(\Wfun{\bar{u}}{\bar{v}})> \alpha$
	for an edge $\bar{u}\bar{v}\in E_G$. Then the edge colouring $\delta$ defined as
	follows  is a pseudo-RS-colouring:
	\begin{align*}
		\delta(uv) :=
		\begin{cases}
			\red{} , &\text{ if } \nu(\Wfun{u}{v}) > \alpha \,,\\
			\gold{}, &\text{ if } \alpha\geq \nu(\Wfun{u}{v})\geq -\alpha\,,
			\\
			\blue{}, &\text{ if } \nu(\Wfun{u}{v}) < -\alpha\,.
		\end{cases}
	\end{align*}
\end{lemma}
\begin{proof}
	By \Cref{rem:WZ}, we have $\nu(\Wfun{\sigma u}{\sigma v}) = -\nu(\Wfun{u}{v})$,
	which gives \cref{it:symmetryRedBlue} in \Cref{def:pseudoRScol}.
	\Cref{it:containsRedBlue} follows from $\delta(\bar{u}\bar{v})=\red$ and $\delta(\sigma\bar{u}\sigma\bar{v})=\blue$.
	Taking thresholds $-\alpha-\varepsilon$ and $\alpha$ for a sufficiently small $\varepsilon>0$
	in \cite[Theorem~2.8]{movableGraphs} gives \cref{it:gold2blue} and \cref{it:gold2red} respectively.
\end{proof}
If a pseudo-RS-colouring can be obtained from a valuation as in \Cref{lem:valuationToPseudoRS},
we call it \emph{active} w.r.t.\ the algebraic motion $\motion$.
Active pseudo-RS-colourings are actually RS-colourings.
\begin{lemma}\label{lem:activeIsRS}
	Let $\motion$ be a reflection-symmetric algebraic motion of $(G,p)$.
	Any pseudo-RS-colouring $\delta$ of $G$ that is active w.r.t.\ $\motion$
	is an RS-colouring. Particularly, if $\delta$ has an almost red-blue cycle,
	then it has a certificate which is active w.r.t.\ $\motion$.  
\end{lemma}
\begin{proof}
	If there is no almost red-blue cycle in $\delta$, we are done.
	Suppose that $C$ is an almost red-blue cycle in $\delta$ with $P$ being the set of red and blue edges.
	Let $\bar{u}\bar{v}$ be the gold edge.
	For $t\in [0,1]$, let $p_t$ be pair-wise distinct elements of $\motion$.
	Suppose that every two red, resp.\ blue, edges in $P$ keep their mutual angle along $\motion$.
	Namely, for every $t$ there is a rotation matrix $\beta^\red_t$ such that $p_t(v)-p_t(u) = \beta^\red_t\cdot (p(v)-p(u))$ for all red $uv\in C$,
	and similarly matrices $\beta^\blue_t$ and $\beta^\gold_t$ with analogous properties.
	For all $t\in [0,1]$ we have
	\begin{align*}
		0 = \sum_{(u,v)\in C}(p_t(v)-p_t(u))
		  = \beta^\gold_t\cdot(p(\bar{v})-p(\bar{u})) + \beta^\blue_t s_\blue + \beta^\red_t s_\red\,,
	\end{align*}
	where
	\begin{align*}
		s_\blue = \sum_{\substack{(u,v)\in C \\ \delta(uv)=\blue}}(p(v)-p(u))
		\qquad \text{and} \qquad
		s_\red = \sum_{\substack{(u,v)\in C \\ \delta(uv)=\red}}(p(v)-p(u))
	\end{align*}
	are independent of $t$.
	Hence,
	\begin{align*}
		0 \neq p(\bar{u})-p(\bar{v}) =
		\left(\beta^\gold_t\right)^{-1}\beta^\blue_t s_\blue + \left(\beta^\gold_t\right)^{-1}\beta^\red_t s_\red
	\end{align*}
	for all $t$,
	which is only possible if at least one of $\left(\beta^\gold_t\right)^{-1}\beta^\blue_t$ or
	$\left(\beta^\gold_t\right)^{-1}\beta^\red_t$ is independent of~$t$.\footnote{
	Let $\beta_t$ and $\beta'_t$ be rotation matrices (possibly) depending on $t\in [0,1]$
	and $s,s'\in\RR^2$. Assume that the sum $\beta_t s + \beta'_t s'$ is a constant non-zero vector independent of $t$.
	The vectors $s,s'$ cannot be both non-zero, otherwise the sum is zero.
	If $s$, resp.\ $s'$, is zero, then $\beta'_t$, resp.\ $\beta$, is independent of $t$.  
	In the remaining case when both $s,s'$ are non-zero, both $\beta_t$ and $\beta'_t$ are independent of $t$:
	The angle between $\beta_t s$ and $\beta'_t s'$ cannot vary with $t$, since the norm of the sum is constant.
	Since also the angle of the sum is constant,
	the angle between $\beta_t s$ and $\beta'_t s'$ must be independent of $t$,
	particularly, $\beta_t$ and $\beta'_t$ are independent of $t$.}
	But that implies that all blue or all red edges keep their angle
	along the flex with the gold edge~$\bar{u}\bar{v}$,
	which contradicts that they have the same value under the valuation~$\nu$ used to construct $\delta$:
	if edges $u_1v_1,u_2v_2$ keep their mutual angle $\varphi$,
	then $\nu(\Wfun{u_1}{v_1})=\nu(e^{\ci\varphi}\Wfun{u_2}{v_2})=\nu(e^{\ci\varphi}) + \nu(\Wfun{u_2}{v_2})=\nu(\Wfun{u_2}{v_2})$.
	Therefore, there are two (w.l.o.g.) blue edges $e_1,e_2$ of $C$ whose
	mutual angle changes along the flex, so by \Cref{lem:nonConstAngle2motionWithValuation} we can construct
	a valuation that distinguishes them and yields an active pseudo-RS-colouring $\delta'$ with $\delta'(e_1)\neq\delta'(e_2)$.
\end{proof}

Combining the lemmas, we can show that a reflection-symmetric flex implies the existence of an RS-colouring.
\begin{theorem}\label{thm:flex2RScolouring}
	If $(G,p)$ is reflection-symmetric flexible,
	then $G$ has an RS-colouring.
\end{theorem}
\begin{proof}
	Since there must be a pair of edges whose angle is not constant along a reflection-symmetric flex,
	there exists an algebraic motion $\motion \subset \mathcal{V}_s(G,p)$
	and valuation $\nu$ on $\CC(\motion)$ by \Cref{lem:nonConstAngle2motionWithValuation}.
	Additionally,
	the set $S:=\{0\} \cup \{\left|\nu (\Wfun{u}{v})\right| \colon uv \in E_G\}$
	has at least two elements.
	We fix $\alpha\in S\setminus \{\max S\}$ to obtain
	an active pseudo-RS-colouring $\delta$ by \Cref{lem:valuationToPseudoRS}.
	By \Cref{lem:activeIsRS}, it is an RS-colouring.
\end{proof}

We note that the necessary condition on the existence of a reflection-symmetric flex is not strong enough
in the sense that there are examples of RS-colourings that cannot be active for any motion,
see \Cref{fig:nonValuationRScolouring}.
In order to strengthen the necessary condition, we need the following concept,
which exploits the idea of \cite[Section~3]{movableGraphs}.
\begin{definition}
	Let $G$ be reflection-symmetric graph with reflection $\sigma$.
	For a reflection-symmetric graph $H$ such that $G$ is a spanning subgraph of $H$,
	let $\pairsGold{H}$ be the set of all pairs $\{u,v\}$ such that $u$ is an invariant vertex,
	$v(\sigma v)\in E_G$ and there is
	a path from $u$ to $v$ that is gold in all RS-colourings of $H$.
	If $G_0,\dots,G_n$ is a sequence of reflection-symmetric graphs with reflection $\sigma$ such that
	\begin{enumerate}
	  \item $G=G_0$,
	  \item $G_i=(V_G,E_{G_{i-1}}\cup \pairsGold{G_{i-1}})$ for $i \in \{1,\ldots, n\}$, and
	  \item $\pairsGold{G_{n}}=\emptyset$,
	\end{enumerate}
	then $G_n$ is called the gold-closure of $G$, denoted by $\clGold{G}$. 
\end{definition}

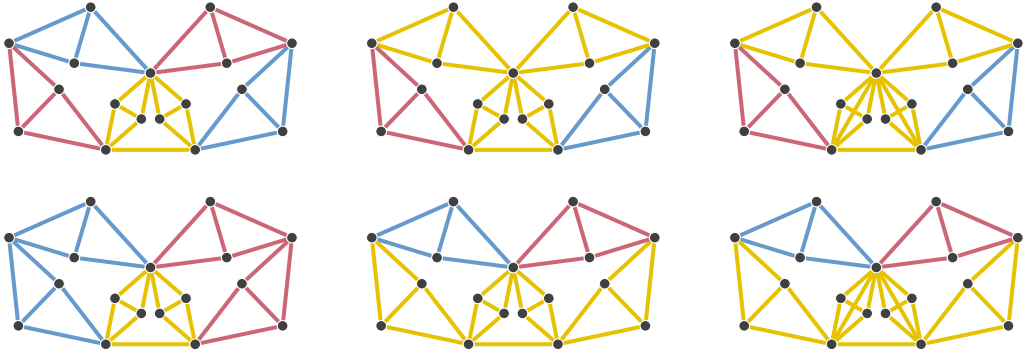
\begin{figure}[ht]
	\centering
	\begin{tikzpicture}
		  	\node[gvertex] (a) at (0.587785, -0.809017) {};
			\node[gvertex] (b) at (-0.587785, -0.809017) {};
			\draw[gedge] (a)edge(b);
		\begin{scope}[rotate around={60:(b)}]
		  	\node[gvertex] (1) at (0.587785, -0.809017) {};
			\node[gvertex] (2) at (0.951057, 0.309017) {};
			\node[gvertex] (3) at (0., 1.) {};
			\node[gvertex] (4) at (-0.951057, 0.309017) {};
			\node[gvertex] (5) at (-0.587785, -0.809017) {};
			
			\node[gvertex] (6) at (0.2, 0.125) {};
			\node[gvertex] (7) at (-0.2, 0.125) {};
			\node[gvertex] (8) at ($1/2*(1)+1/2*(5)+(0,0.2)$) {};
			\node[gvertex] (9) at ($1/2*(1)+1/2*(5)+(0,-0.2)$) {};
			\draw[gedge] (1)edge(8) (1)edge(9) (5)edge(8) (5)edge(9) (8)edge(9);
			\draw[bedge] (2)edge(3) (2)edge(1) (1)edge(6) (2)edge(6) (3)edge(6);
			\draw[redge] (4)edge(3) (5)edge(4) (4)edge(7) (5)edge(7) (3)edge(7);
		\end{scope}
		\begin{scope}[rotate around={-60:(a)}]
		  	\node[gvertex] (1) at (0.587785, -0.809017) {};
			\node[gvertex] (2) at (0.951057, 0.309017) {};
			\node[gvertex] (3) at (0., 1.) {};
			\node[gvertex] (4) at (-0.951057, 0.309017) {};
			\node[gvertex] (5) at (-0.587785, -0.809017) {};
			
			\node[gvertex] (6) at (0.2, 0.125) {};
			\node[gvertex] (7) at (-0.2, 0.125) {};
			\node[gvertex] (8) at ($1/2*(1)+1/2*(5)+(0,0.2)$) {};
			\node[gvertex] (9) at ($1/2*(1)+1/2*(5)+(0,-0.2)$) {};
			\draw[gedge] (1)edge(8) (1)edge(9) (5)edge(8) (5)edge(9) (8)edge(9);
			\draw[bedge] (2)edge(3) (2)edge(1) (1)edge(6) (2)edge(6) (3)edge(6);
			\draw[redge] (4)edge(3) (5)edge(4) (4)edge(7) (5)edge(7) (3)edge(7);
		\end{scope}
	\end{tikzpicture}
	\qquad
	\begin{tikzpicture}
		  	\node[gvertex] (a) at (0.587785, -0.809017) {};
			\node[gvertex] (b) at (-0.587785, -0.809017) {};
			\draw[gedge] (a)edge(b);
		\begin{scope}[rotate around={60:(b)}]
		  	\node[gvertex] (1) at (0.587785, -0.809017) {};
			\node[gvertex] (2) at (0.951057, 0.309017) {};
			\node[gvertex] (3) at (0., 1.) {};
			\node[gvertex] (4) at (-0.951057, 0.309017) {};
			\node[gvertex] (5) at (-0.587785, -0.809017) {};
			
			\node[gvertex] (6) at (0.2, 0.125) {};
			\node[gvertex] (7) at (-0.2, 0.125) {};
			\node[gvertex] (8) at ($1/2*(1)+1/2*(5)+(0,0.2)$) {};
			\node[gvertex] (9) at ($1/2*(1)+1/2*(5)+(0,-0.2)$) {};
			\draw[gedge] (1)edge(8) (1)edge(9) (5)edge(8) (5)edge(9) (8)edge(9);
			\draw[gedge] (2)edge(3) (2)edge(1) (1)edge(6) (2)edge(6) (3)edge(6);
			\draw[redge] (4)edge(3) (5)edge(4) (4)edge(7) (5)edge(7) (3)edge(7);
		\end{scope}
		\begin{scope}[rotate around={-60:(a)}]
		  	\node[gvertex] (1) at (0.587785, -0.809017) {};
			\node[gvertex] (2) at (0.951057, 0.309017) {};
			\node[gvertex] (3) at (0., 1.) {};
			\node[gvertex] (4) at (-0.951057, 0.309017) {};
			\node[gvertex] (5) at (-0.587785, -0.809017) {};
			
			\node[gvertex] (6) at (0.2, 0.125) {};
			\node[gvertex] (7) at (-0.2, 0.125) {};
			\node[gvertex] (8) at ($1/2*(1)+1/2*(5)+(0,0.2)$) {};
			\node[gvertex] (9) at ($1/2*(1)+1/2*(5)+(0,-0.2)$) {};
			\draw[gedge] (1)edge(8) (1)edge(9) (5)edge(8) (5)edge(9) (8)edge(9);
			\draw[bedge] (2)edge(3) (2)edge(1) (1)edge(6) (2)edge(6) (3)edge(6);
			\draw[gedge] (4)edge(3) (5)edge(4) (4)edge(7) (5)edge(7) (3)edge(7);
		\end{scope}
	\end{tikzpicture}
	\qquad
	\begin{tikzpicture}
		  	\node[gvertex] (a) at (0.587785, -0.809017) {};
			\node[gvertex] (b) at (-0.587785, -0.809017) {};
			\draw[gedge] (a)edge(b);
		\begin{scope}[rotate around={60:(b)}]
		  	\node[gvertex] (1) at (0.587785, -0.809017) {};
			\node[gvertex] (2) at (0.951057, 0.309017) {};
			\node[gvertex] (3) at (0., 1.) {};
			\node[gvertex] (4) at (-0.951057, 0.309017) {};
			\node[gvertex] (5) at (-0.587785, -0.809017) {};
			
			\node[gvertex] (6) at (0.2, 0.125) {};
			\node[gvertex] (7) at (-0.2, 0.125) {};
			\node[gvertex] (8) at ($1/2*(1)+1/2*(5)+(0,0.2)$) {};
			\node[gvertex] (9) at ($1/2*(1)+1/2*(5)+(0,-0.2)$) {};
			\draw[gedge] (1)edge(8) (1)edge(9) (5)edge(8) (5)edge(9) (8)edge(9);
			\draw[gedge] (2)edge(3) (2)edge(1) (1)edge(6) (2)edge(6) (3)edge(6);
			\draw[redge] (4)edge(3) (5)edge(4) (4)edge(7) (5)edge(7) (3)edge(7);
			\draw[gedge] (1)edge(5);
		\end{scope}
		\begin{scope}[rotate around={-60:(a)}]
		  	\node[gvertex] (1) at (0.587785, -0.809017) {};
			\node[gvertex] (2) at (0.951057, 0.309017) {};
			\node[gvertex] (3) at (0., 1.) {};
			\node[gvertex] (4) at (-0.951057, 0.309017) {};
			\node[gvertex] (5) at (-0.587785, -0.809017) {};
			
			\node[gvertex] (6) at (0.2, 0.125) {};
			\node[gvertex] (7) at (-0.2, 0.125) {};
			\node[gvertex] (8) at ($1/2*(1)+1/2*(5)+(0,0.2)$) {};
			\node[gvertex] (9) at ($1/2*(1)+1/2*(5)+(0,-0.2)$) {};
			\draw[gedge] (1)edge(8) (1)edge(9) (5)edge(8) (5)edge(9) (8)edge(9);
			\draw[bedge] (2)edge(3) (2)edge(1) (1)edge(6) (2)edge(6) (3)edge(6);
			\draw[gedge] (4)edge(3) (5)edge(4) (4)edge(7) (5)edge(7) (3)edge(7);
			\draw[gedge] (1)edge(5);
		\end{scope}
	\end{tikzpicture}
	
	\vspace{0.5cm}
	\begin{tikzpicture}
		  	\node[gvertex] (a) at (0.587785, -0.809017) {};
			\node[gvertex] (b) at (-0.587785, -0.809017) {};
			\draw[gedge] (a)edge(b);
		\begin{scope}[rotate around={60:(b)}]
		  	\node[gvertex] (1) at (0.587785, -0.809017) {};
			\node[gvertex] (2) at (0.951057, 0.309017) {};
			\node[gvertex] (3) at (0., 1.) {};
			\node[gvertex] (4) at (-0.951057, 0.309017) {};
			\node[gvertex] (5) at (-0.587785, -0.809017) {};
			
			\node[gvertex] (6) at (0.2, 0.125) {};
			\node[gvertex] (7) at (-0.2, 0.125) {};
			\node[gvertex] (8) at ($1/2*(1)+1/2*(5)+(0,0.2)$) {};
			\node[gvertex] (9) at ($1/2*(1)+1/2*(5)+(0,-0.2)$) {};
			\draw[gedge] (1)edge(8) (1)edge(9) (5)edge(8) (5)edge(9) (8)edge(9);
			\draw[bedge] (2)edge(3) (2)edge(1) (1)edge(6) (2)edge(6) (3)edge(6);
			\draw[bedge] (4)edge(3) (5)edge(4) (4)edge(7) (5)edge(7) (3)edge(7);
		\end{scope}
		\begin{scope}[rotate around={-60:(a)}]
		  	\node[gvertex] (1) at (0.587785, -0.809017) {};
			\node[gvertex] (2) at (0.951057, 0.309017) {};
			\node[gvertex] (3) at (0., 1.) {};
			\node[gvertex] (4) at (-0.951057, 0.309017) {};
			\node[gvertex] (5) at (-0.587785, -0.809017) {};
			
			\node[gvertex] (6) at (0.2, 0.125) {};
			\node[gvertex] (7) at (-0.2, 0.125) {};
			\node[gvertex] (8) at ($1/2*(1)+1/2*(5)+(0,0.2)$) {};
			\node[gvertex] (9) at ($1/2*(1)+1/2*(5)+(0,-0.2)$) {};
			\draw[gedge] (1)edge(8) (1)edge(9) (5)edge(8) (5)edge(9) (8)edge(9);
			\draw[redge] (2)edge(3) (2)edge(1) (1)edge(6) (2)edge(6) (3)edge(6);
			\draw[redge] (4)edge(3) (5)edge(4) (4)edge(7) (5)edge(7) (3)edge(7);
		\end{scope}
	\end{tikzpicture}
	\qquad
	\begin{tikzpicture}
		  	\node[gvertex] (a) at (0.587785, -0.809017) {};
			\node[gvertex] (b) at (-0.587785, -0.809017) {};
			\draw[gedge] (a)edge(b);
		\begin{scope}[rotate around={60:(b)}]
		  	\node[gvertex] (1) at (0.587785, -0.809017) {};
			\node[gvertex] (2) at (0.951057, 0.309017) {};
			\node[gvertex] (3) at (0., 1.) {};
			\node[gvertex] (4) at (-0.951057, 0.309017) {};
			\node[gvertex] (5) at (-0.587785, -0.809017) {};
			
			\node[gvertex] (6) at (0.2, 0.125) {};
			\node[gvertex] (7) at (-0.2, 0.125) {};
			\node[gvertex] (8) at ($1/2*(1)+1/2*(5)+(0,0.2)$) {};
			\node[gvertex] (9) at ($1/2*(1)+1/2*(5)+(0,-0.2)$) {};
			\draw[gedge] (1)edge(8) (1)edge(9) (5)edge(8) (5)edge(9) (8)edge(9);
			\draw[bedge] (2)edge(3) (2)edge(1) (1)edge(6) (2)edge(6) (3)edge(6);
			\draw[gedge] (4)edge(3) (5)edge(4) (4)edge(7) (5)edge(7) (3)edge(7);
		\end{scope}
		\begin{scope}[rotate around={-60:(a)}]
		  	\node[gvertex] (1) at (0.587785, -0.809017) {};
			\node[gvertex] (2) at (0.951057, 0.309017) {};
			\node[gvertex] (3) at (0., 1.) {};
			\node[gvertex] (4) at (-0.951057, 0.309017) {};
			\node[gvertex] (5) at (-0.587785, -0.809017) {};
			
			\node[gvertex] (6) at (0.2, 0.125) {};
			\node[gvertex] (7) at (-0.2, 0.125) {};
			\node[gvertex] (8) at ($1/2*(1)+1/2*(5)+(0,0.2)$) {};
			\node[gvertex] (9) at ($1/2*(1)+1/2*(5)+(0,-0.2)$) {};
			\draw[gedge] (1)edge(8) (1)edge(9) (5)edge(8) (5)edge(9) (8)edge(9);
			\draw[gedge] (2)edge(3) (2)edge(1) (1)edge(6) (2)edge(6) (3)edge(6);
			\draw[redge] (4)edge(3) (5)edge(4) (4)edge(7) (5)edge(7) (3)edge(7);
		\end{scope}
	\end{tikzpicture}
	\qquad
	\begin{tikzpicture}
		  	\node[gvertex] (a) at (0.587785, -0.809017) {};
			\node[gvertex] (b) at (-0.587785, -0.809017) {};
			\draw[gedge] (a)edge(b);
		\begin{scope}[rotate around={60:(b)}]
		  	\node[gvertex] (1) at (0.587785, -0.809017) {};
			\node[gvertex] (2) at (0.951057, 0.309017) {};
			\node[gvertex] (3) at (0., 1.) {};
			\node[gvertex] (4) at (-0.951057, 0.309017) {};
			\node[gvertex] (5) at (-0.587785, -0.809017) {};
			
			\node[gvertex] (6) at (0.2, 0.125) {};
			\node[gvertex] (7) at (-0.2, 0.125) {};
			\node[gvertex] (8) at ($1/2*(1)+1/2*(5)+(0,0.2)$) {};
			\node[gvertex] (9) at ($1/2*(1)+1/2*(5)+(0,-0.2)$) {};
			\draw[gedge] (1)edge(8) (1)edge(9) (5)edge(8) (5)edge(9) (8)edge(9);
			\draw[bedge] (2)edge(3) (2)edge(1) (1)edge(6) (2)edge(6) (3)edge(6);
			\draw[gedge] (4)edge(3) (5)edge(4) (4)edge(7) (5)edge(7) (3)edge(7);
			\draw[gedge] (1)edge(5);
		\end{scope}
		\begin{scope}[rotate around={-60:(a)}]
		  	\node[gvertex] (1) at (0.587785, -0.809017) {};
			\node[gvertex] (2) at (0.951057, 0.309017) {};
			\node[gvertex] (3) at (0., 1.) {};
			\node[gvertex] (4) at (-0.951057, 0.309017) {};
			\node[gvertex] (5) at (-0.587785, -0.809017) {};
			
			\node[gvertex] (6) at (0.2, 0.125) {};
			\node[gvertex] (7) at (-0.2, 0.125) {};
			\node[gvertex] (8) at ($1/2*(1)+1/2*(5)+(0,0.2)$) {};
			\node[gvertex] (9) at ($1/2*(1)+1/2*(5)+(0,-0.2)$) {};
			\draw[gedge] (1)edge(8) (1)edge(9) (5)edge(8) (5)edge(9) (8)edge(9);
			\draw[gedge] (2)edge(3) (2)edge(1) (1)edge(6) (2)edge(6) (3)edge(6);
			\draw[redge] (4)edge(3) (5)edge(4) (4)edge(7) (5)edge(7) (3)edge(7);
			\draw[gedge] (1)edge(5);
		\end{scope}
	\end{tikzpicture}
	\caption{All RS-colourings of a graph modulo conjugation (left and middle columns).
	The gold-closure of the graph is on the right with its two RS-colourings.
	The middle RS-colourings yield flexes using \Cref{thm:noAlmostRBcycle2flex},
	while the left two cannot be active w.r.t.\ any motion,
	since they are not restrictions of the RS-colourings of the gold-closure
	(see \Cref{lem:goldClosure}).  
	}
	\label{fig:nonValuationRScolouring}
\end{figure}

\begin{lemma}\label{lem:goldClosure}
	Any reflection-symmetric algebraic motion $\motion$ of a reflection-symmetric framework $(G,p)$
	is a reflection-symmetric algebraic motion of $(\clGold{G},p)$.
	Moreover, the active RS-colourings of $G$ w.r.t.\ $\motion$ are exactly the restrictions of
	the active RS-colourings of $\clGold{G}$ and every edge in $E_G\setminus E_{\clGold{G}}$ is gold
	in all RS-colourings of $\clGold{G}$. 
\end{lemma}
\begin{proof}
	In order to prove that a reflection-symmetric algebraic motion $\motion$ of $(G,p)$
	is also a reflection-symmetric algebraic motion of $(\clGold{G},p)$,
	it suffices to show that $\|p'(u)-p'(v)\|=\|p(u)-p(v)\|$
	for all $uv\in \pairsGold{G}$ and $p'\in \motion$,
	then the statement follows by induction.
	Let $uv\in \pairsGold{G}$ be such that $u=\sigma u$ and $v\sigma v\in E_G$.
	Let $P$ be a path from $u$ to $v$ that is gold for all RS-colourings of $G$.
	All edges in $P$ keep their mutual angles along $\motion$,
	otherwise, by \Cref{it:flex2motion} of \Cref{lem:nonConstAngle2motionWithValuation}, and
	\Cref{lem:valuationToPseudoRS,lem:activeIsRS},
	there would be an active RS-coloring distinguishing them, which is not possible by the assumption.
	Therefore, the distance between $u$ and $v$ is constant along the motion $\motion$.
	This distance is non-zero since $p(u)$ is on the $x$-axis as $u$ is an invariant vertex,
	while $v$ is not on the $x$-axis as $v(\sigma v)$ is an edge.
	Since $P$ is gold in all RS-colourings of $G$,
	it is gold also in all RS-colourings of $\clGold{G}$.
	Hence, the edge $uv$ is also gold in all RS-colourings of $\clGold{G}$.
	The fact the active RS-colourings of $G$ w.r.t.\ $\motion$ are exactly the restrictions of
	the active RS-colourings of $\clGold{G}$ w.r.t.\ $\motion$ holds since algebraic motion is the same for both graphs. 
\end{proof}
As a direct consequence, we obtain the following necessary condition on the existence of a reflection-symmetric flex.
\begin{corollary}
	\label{cor:flex2RScolouringOfGoldClosure}
	If $(G,p)$ is reflection-symmetric flexible,
	then $\clGold{G}$ has an RS-colouring.
\end{corollary}

\section{Constructions of reflection-symmetric flexes}
\label{sec:sufficient}

In this section we show how to construct a flex if we have a special RS-colouring
or a pair of them with certain properties.
When we have no almost \red{}-\blue\ cycle, we can give a construction of a reflection-symmetric framework with its flex.
\begin{theorem}
	\label{thm:noAlmostRBcycle2flex}
	Let $G$ be a reflection-symmetric graph with an RS-colouring
	that has no almost \red{}-\blue{} cycle.
	Then there exists a reflection-symmetric flexible framework $(G,p)$. 
\end{theorem}
\begin{proof}
  We show a grid construction which involves special treatment of the golden edges.
  Let $R_1,\ldots,R_k$ be the connected components of the subgraph of $G$
  obtained by removing all \blue{} edges.
  Let $r_1,\ldots,r_k\in\RR^2$ be distinct points. 
  They serve as basepoints of our grid.
  To assign each vertex to a grid point we define the function $a\colon V_G\rightarrow \RR^2$
  with $a(u)=r_i$ if $u\in R_i$.
  
  In order to deal with the golden part we define the connected components
  of the subgraph of $G$ obtained by removing the \gold{} edges.
  We denote the invariant components by $\bar D_1,\ldots,\bar D_\ell$ (i.e.\ $\bar D_i=\sigma\bar D_i$ for all $i\in\range{1}{\ell}$).
  The components that are not invariant are denoted in a way to respect the symmetry, i.e.~$D_1,\ldots,D_\kappa$
  and $\tilde D_1,\ldots,\tilde D_\kappa$ such that $D_i=\sigma\tilde D_i$ for all $i\in\range{1}{\kappa}$.
  Let $d_1,\ldots,d_\kappa\in\RR^2$ be distinct points such that none lie on the $y$-axis
  and no two are reflection-symmetric.
  Let $\overline{d}_1,\ldots,\overline{d}_\ell\in\RR$ be distinct,
  and let $z$ be defined by 
  \begin{align*}
    z(u)=
    \begin{cases}
      d_i & \text{if } u\in D_i,\\
      \tau d_i & \text{if } u\in\tilde D_i,\\
      (0,\overline{d}_i) & \text{if } u\in\bar D_i.
    \end{cases}
  \end{align*}
  We see that $z(\sigma u)=\tau z(u)$, which we require in the rest of the proof.
  For $t \in [0,2\pi]$, define
  \begin{align*}
	\rott:=  	
  	\begin{bmatrix}
		\cos t & -\sin t \\ 
		\sin t & \cos t 
	\end{bmatrix}.
  \end{align*}
  We now obtain for each $t \in [0,2\pi]$ a realisation $p_t$ of $G$ given by
	\begin{equation*}
		p_t(u) :=\rott a(u)	+ \rotmt \tau a(\sigma u) + z(u) \, .
	\end{equation*}
	Using $\rott \tau = \tau \rotmt$, we check that each realisation is reflection-symmetric:
	\begin{equation*}
		p_t(\sigma u) = \rott a(\sigma u) + \rotmt \tau a(u)+ z(\sigma u)
			= \tau\rotmt \tau a(\sigma u) + \tau \rott a(u)+ \tau z(u) = \tau p_t(u)\,.
	\end{equation*}
		
	We now wish to show that for every edge $uv\in E_G$, the edge length of $uv$ is constant
	for all values of $t \in [0,2\pi]$.
	If $uv$ is a \red{} edge,
	then $a(u)=a(v)$ and also $z(u)=z(v)$,
	hence
	\begin{align*}
	  \| p_t(u)-p_t(v) \|= \left\| \rotmt \tau( a(\sigma u) - a(\sigma v)) \right\| = \|a(\sigma u) - a(\sigma v) \|.
	\end{align*}
	If $uv$ is a \blue{} edge,
	then $\sigma(uv)$ is \red{} and $\| p_t(u)-p_t(v) \|$ is constant by the reflection-symmetry.
	If $uv$ is a \gold{} edge,
	then $a(u)=a(v)$ and also $a(\sigma u)=a(\sigma v)$,
	hence
	\begin{align*}
	  \| p_t(u)-p_t(v) \|= \|z(u) - z(v) \|.
	\end{align*}
	
	Assume now that $p_0(u) = p_0(v)$ for some edge $uv \in E_G$.
	If $uv$ is \red{}, then $a(\sigma u) = a(\sigma v)$ by the computation above,
	which implies $\sigma u$ and $\sigma v$ are in the same $R_i$ by the assumption that $r_1,\ldots,r_k$ are distinct.
	But this is a contradiction, since the end vertices of the blue edge $\sigma(uv)$ cannot be connected by a path containing only red and gold edges (see \cref{it:gold2red} in \Cref{def:pseudoRScol}).
	By symmetry, $uv$ cannot be blue either.
	If $uv$ is \gold{}, then $z(u)=z(v)$, which is not possible as there are no almost red-blue cycles by assumption.
	Thus $p_0(u) \neq p_0(v)$.

	The flex $p_t$ is non-trivial as there exists some blue edge $e$ that changes its angle with the red edge $\sigma e$.
	Hence, $(G,p_0)$ is a reflection-symmetric flexible framework. 
\end{proof}

Similarly to the general case \cite{flexibleLabelings},
we can view this construction as placing vertices on a grid.
\Cref{fig:grid} shows the basic grid construction used when there are no golden edges, and also how this grid can be altered to allow for golden edges.
While the figure use a small perturbation $z$ (as defined in the proof of \Cref{thm:noAlmostRBcycle2flex}), the golden part might be of any size as $z$ can be chosen to be significantly larger.
The RS-colourings of \Cref{fig:pseudoRScol} lead to the constructions in \Cref{fig:gridexample}.
A more complicated example can be seen in \Cref{fig:gridexample2}.

\begin{figure}[ht]
	\centering
	\begin{tikzpicture}
		\pgfmathparse{3}\let\lim=\pgfmathresult
		\draw[sym] (0,-0.25)--(0,4.5);
		\begin{scope}
			\pgfmathparse{38}\let\w=\pgfmathresult
			\foreach \i [count=\c] in {0,1,2,...,3}
			{
				\draw[gridl] (180-\w:\i) node[below left,labelsty] {$\sigma r_\c$} --++(\w:\lim);
				\draw[gridl] (\w:\i) node[below right,labelsty] {$r_\c$} --++(180-\w:\lim);
			}
		\end{scope}
		\begin{scope}[opacity=0.25]
			\pgfmathparse{45}\let\w=\pgfmathresult
			\foreach \i [count=\c] in {0,1,2,...,3}
			{
				\draw[gridl] (180-\w:\i) --++(\w:\lim);
				\draw[gridl] (\w:\i) --++(180-\w:\lim);
			}
		\end{scope}
		\begin{scope}[opacity=0.08]
			\pgfmathparse{55}\let\w=\pgfmathresult
			\foreach \i [count=\c] in {0,1,2,...,3}
			{
				\draw[gridl] (180-\w:\i) --++(\w:\lim);
				\draw[gridl] (\w:\i) --++(180-\w:\lim);
			}
		\end{scope}
	\end{tikzpicture}
	\qquad
	\begin{tikzpicture}[spy scope={magnification=19, size=1cm},
			every spy in node/.style={
			magnifying glass, circular drop shadow,
			fill=white, draw, ultra thick, cap=round}
		]
		\pgfmathparse{40}\let\w=\pgfmathresult
		\pgfmathparse{3}\let\lim=\pgfmathresult
		\draw[sym] (0,0.25)--(0,4.25);
		\coordinate (d1) at (-30:0.01);
		\coordinate (d2) at (90:0.01);
		\coordinate (d3) at (210:0.01);
		\foreach \d in {d1,d2,d3}
		{
			\begin{scope}[shift=(\d)]
				\foreach \i in {0,1,2,...,3}
				{
					\draw[gridl2] (180-\w:\i)--++(\w:\lim);
					\draw[gridl2] (\w:\i)--++(180-\w:\lim);
				}
				\foreach \i in {0,1,2,3}
				{
					\foreach \j in {0,1,2,3}
					{
						\coordinate (d\d-\i-\j) at ($(\w:\i)+(180-\w:\j)$);
					}
				}
			\end{scope}
		}
		\foreach \i in {0,1,2,3}
		{
			\foreach \j in {0,1,2,3}
			{
				\draw[gridl2,colG,opacity=0.75] (dd1-\i-\j)edge(dd2-\i-\j) (dd2-\i-\j)edge(dd3-\i-\j) (dd3-\i-\j)edge(dd1-\i-\j);
				\node[fvertexs] at (dd1-\i-\j)  {};
				\node[fvertexs] at (dd2-\i-\j)  {};
				\node[fvertexs] at (dd3-\i-\j)  {};
			}
		}
		\spy on (0,0) in node;
	\end{tikzpicture}
	\caption{The basic grid construction (left) with some instances of the grid flex and a grid with zoom on the golden edges (right).}\label{fig:grid}
\end{figure}
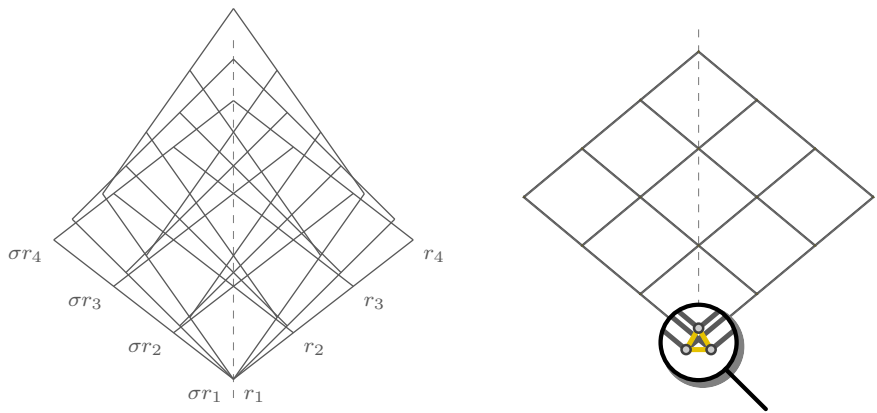

\begin{figure}[ht]
	\begin{center}
		\begin{tabular}{ccccc}
			\begin{tikzpicture}[scale=0.8]
				\draw[sym] (0,-0.2) -- (0,2.3);
				\node[gvertex] (1) at (-0.51, 0.) {};
				\node[gvertex] (2) at (0.51, 0.) {};
				\node[gvertexh] (3) at (0., 0.74) {};
				\node[gvertex] (4) at (-1.35, 0.5) {};
				\node[gvertex] (5) at (-0.76, 1.5) {};
				\node[gvertexh] (6) at (0., 2.1) {};
				\node[gvertex] (7) at (0.76, 1.5) {};
				\node[gvertex] (8) at (1.35, 0.5) {};
				\draw[gedge] (1)edge(2);
				\draw[gedge] (1)edge(3);
				\draw[bedge] (1)edge(4);
				\draw[gedge] (2)edge(3);
				\draw[redge] (2)edge(8);
				\draw[bedge] (3)edge(5);
				\draw[redge] (3)edge(7);
				\draw[gedge] (4)edge(5);
				\draw[bedge] (5)edge(6);
				\draw[redge] (6)edge(7);
				\draw[gedge] (7)edge(8);
			\end{tikzpicture}
			&
			\begin{tikzpicture}[scale=0.8]
				\draw[sym] (0,-0.2) -- (0,2.3);
				\node[gvertex] (1) at (-0.51, 0.) {};
				\node[gvertex] (2) at (0.51, 0.) {};
				\node[gvertex] (3) at (0., 0.74) {};
				\node[gvertex] (4) at (-1.35, 0.5) {};
				\node[gvertex] (5) at (-0.76, 1.5) {};
				\node[gvertex] (6) at (0., 2.1) {};
				\node[gvertex] (7) at (0.76, 1.5) {};
				\node[gvertex] (8) at (1.35, 0.5) {};
				\draw[gedge] (1)edge(2);
				\draw[gedge] (1)edge(3);
				\draw[bedge] (1)edge(4);
				\draw[gedge] (2)edge(3);
				\draw[redge] (2)edge(8);
				\draw[bedge] (3)edge(5);
				\draw[redge] (3)edge(7);
				\draw[gedge] (4)edge(5);
				\draw[redge] (5)edge(6);
				\draw[bedge] (6)edge(7);
				\draw[gedge] (7)edge(8);
			\end{tikzpicture}
			&
			\begin{tikzpicture}[scale=0.8]
				\draw[sym] (0,-0.2) -- (0,2.3);
				\node[gvertexh] (1) at (-0.51, 0.) {};
				\node[gvertexh] (2) at (0.51, 0.) {};
				\node[gvertex] (3) at (0., 0.74) {};
				\node[gvertex] (4) at (-1.35, 0.5) {};
				\node[gvertexh] (5) at (-0.76, 1.5) {};
				\node[gvertex] (6) at (0., 2.1) {};
				\node[gvertexh] (7) at (0.76, 1.5) {};
				\node[gvertex] (8) at (1.35, 0.5) {};
				\draw[gedge] (1)edge(2);
				\draw[gedge] (1)edge(3);
				\draw[bedge] (1)edge(4);
				\draw[gedge] (2)edge(3);
				\draw[redge] (2)edge(8);
				\draw[gedge] (3)edge(5);
				\draw[gedge] (3)edge(7);
				\draw[bedge] (4)edge(5);
				\draw[gedge] (5)edge(6);
				\draw[gedge] (6)edge(7);
				\draw[redge] (7)edge(8);
			\end{tikzpicture}
			&
			\begin{tikzpicture}[scale=0.8]
				\draw[sym] (0,-0.2) -- (0,2.3);
				\node[gvertex] (1) at (-0.51, 0.) {};
				\node[gvertex] (2) at (0.51, 0.) {};
				\node[gvertexh] (3) at (0., 0.74) {};
				\node[gvertexh] (4) at (-1.35, 0.5) {};
				\node[gvertex] (5) at (-0.76, 1.5) {};
				\node[gvertexh] (6) at (0., 2.1) {};
				\node[gvertex] (7) at (0.76, 1.5) {};
				\node[gvertexh] (8) at (1.35, 0.5) {};
				\draw[gedge] (1)edge(2);
				\draw[gedge] (1)edge(3);
				\draw[gedge] (1)edge(4);
				\draw[gedge] (2)edge(3);
				\draw[gedge] (2)edge(8);
				\draw[bedge] (3)edge(5);
				\draw[redge] (3)edge(7);
				\draw[bedge] (4)edge(5);
				\draw[bedge] (5)edge(6);
				\draw[redge] (6)edge(7);
				\draw[redge] (7)edge(8);
			\end{tikzpicture}
			&
			\begin{tikzpicture}[scale=0.8]
				\draw[sym] (0,-0.2) -- (0,2.3);
				\node[gvertex] (1) at (-0.51, 0.) {};
				\node[gvertex] (2) at (0.51, 0.) {};
				\node[gvertexh] (3) at (0., 0.74) {};
				\node[gvertexh] (4) at (-1.35, 0.5) {};
				\node[gvertex] (5) at (-0.76, 1.5) {};
				\node[gvertex] (6) at (0., 2.1) {};
				\node[gvertex] (7) at (0.76, 1.5) {};
				\node[gvertexh] (8) at (1.35, 0.5) {};
				\draw[gedge] (1)edge(2);
				\draw[gedge] (1)edge(3);
				\draw[gedge] (1)edge(4);
				\draw[gedge] (2)edge(3);
				\draw[gedge] (2)edge(8);
				\draw[bedge] (3)edge(5);
				\draw[redge] (3)edge(7);
				\draw[bedge] (4)edge(5);
				\draw[redge] (5)edge(6);
				\draw[bedge] (6)edge(7);
				\draw[redge] (7)edge(8);
			\end{tikzpicture}
			\\
			\begin{tikzpicture}[scale=0.8]
				\draw[sym] (0,-0.2) -- (0,2.3);
				\foreach \w/\o in {65/0.15,40/1}
				{
					\begin{scope}[opacity=\o]
						\coordinate (a1) at (\w:1);
						\coordinate (p1) at ($(0,0)!(a1)!(0,1)$);
						\coordinate (sa1) at ($(p1)+(p1)-(a1)$);
						\coordinate (d1) at (0,0.75);
						\coordinate (dt1) at (0.5,0);
						\node[fvertex] (1) at ($(0,0)-(dt1)$) {};
						\node[fvertex] (2) at (dt1) {};
						\node[fvertexh] (3) at (d1) {};
						\node[fvertex] (4) at ($(sa1)-(dt1)$) {};
						\node[fvertex] (5) at ($(sa1)+(d1)$) {}; {};
						\node[fvertexh] (6) at (3) {}; {};
						\node[fvertex] (7) at ($(a1)+(d1)$) {}; {};
						\node[fvertex] (8) at ($(a1)+(dt1)$) {};
						\draw[gedge] (1)edge(2);
						\draw[gedge] (1)edge(3);
						\draw[bedge] (1)edge(4);
						\draw[gedge] (2)edge(3);
						\draw[redge] (2)edge(8);
						\draw[bedge] (3)edge(5);
						\draw[redge] (3)edge(7);
						\draw[gedge] (4)edge(5);
						\draw[bedge] (5)edge(6);
						\draw[redge] (6)edge(7);
						\draw[gedge] (7)edge(8);
					\end{scope}
				}
			\end{tikzpicture}
			&
			\begin{tikzpicture}[scale=0.8]
				\draw[sym] (0,-0.2) -- (0,2.3);
				\foreach \w/\o in {65/0.15,40/1}
				{
					\begin{scope}[opacity=\o]
						\coordinate (a1) at (\w:1);
						\coordinate (p1) at ($(0,0)!(a1)!(0,1)$);
						\coordinate (sa1) at ($(p1)+(p1)-(a1)$);
						\coordinate (d1) at (0,0.75);
						\coordinate (dt1) at (0.5,0);
						\node[fvertex] (1) at ($(0,0)-(dt1)$) {};
						\node[fvertex] (2) at (dt1) {};
						\node[fvertex] (3) at (d1) {};
						\node[fvertex] (4) at ($(sa1)-(dt1)$) {};
						\node[fvertex] (5) at ($(sa1)+(d1)$) {}; {};
						\node[fvertex] (6) at ($(a1)+(sa1)+(d1)$) {}; {};
						\node[fvertex] (7) at ($(a1)+(d1)$) {}; {};
						\node[fvertex] (8) at ($(a1)+(dt1)$) {};
						\draw[gedge] (1)edge(2);
						\draw[gedge] (1)edge(3);
						\draw[bedge] (1)edge(4);
						\draw[gedge] (2)edge(3);
						\draw[redge] (2)edge(8);
						\draw[bedge] (3)edge(5);
						\draw[redge] (3)edge(7);
						\draw[gedge] (4)edge(5);
						\draw[redge] (5)edge(6);
						\draw[bedge] (6)edge(7);
						\draw[gedge] (7)edge(8);
					\end{scope}
				}
			\end{tikzpicture}
			&
			\begin{tikzpicture}[scale=0.8]
				\draw[sym] (0,-0.2) -- (0,2.3);
				\foreach \w/\o in {65/0.15,40/1}
				{
					\begin{scope}[opacity=\o]
						\coordinate (a1) at (\w:1);
						\coordinate (p1) at ($(0,0)!(a1)!(0,1)$);
						\coordinate (sa1) at ($(p1)+(p1)-(a1)$);
						\coordinate (d1) at (0,0.75);
						\coordinate (dt1) at (0.5,0);
						\node[fvertexh] (1) at ($(0,0)-(dt1)$) {};
						\node[fvertexh] (2) at (dt1) {};
						\node[fvertex] (3) at (d1) {};
						\node[fvertex] (4) at ($(sa1)-(dt1)$) {};
						\node[fvertexh] (5) at (1) {}; {};
						\node[fvertex] (6) at (0,2) {}; {};
						\node[fvertexh] (7) at (2) {}; {};
						\node[fvertex] (8) at ($(a1)+(dt1)$) {};
						\draw[gedge] (1)edge(2);
						\draw[gedge] (1)edge(3);
						\draw[bedge] (1)edge(4);
						\draw[gedge] (2)edge(3);
						\draw[redge] (2)edge(8);
						\draw[gedge] (3)edge(5);
						\draw[gedge] (3)edge(7);
						\draw[bedge] (4)edge(5);
						\draw[gedge] (5)edge(6);
						\draw[gedge] (6)edge(7);
						\draw[redge] (7)edge(8);
					\end{scope}
				}
			\end{tikzpicture}
			&
			\begin{tikzpicture}[scale=0.8]
				\draw[sym] (0,-0.2) -- (0,2.3);
				\foreach \w/\o in {65/0.15,40/1}
				{
					\begin{scope}[opacity=\o]
						\coordinate (a1) at (\w:1);
						\coordinate (p1) at ($(0,0)!(a1)!(0,1)$);
						\coordinate (sa1) at ($(p1)+(p1)-(a1)$);
						\coordinate (d1) at (0,0.75);
						\coordinate (dt1) at (0.5,0);
						\node[fvertex] (1) at ($(0,0)-(dt1)$) {};
						\node[fvertex] (2) at (dt1) {};
						\node[fvertexh] (3) at (d1) {};
						\node[fvertexh] (4) at (3) {};
						\node[fvertex] (5) at ($(sa1)+(d1)$) {}; {};
						\node[fvertexh] (6) at (3) {}; {};
						\node[fvertex] (7) at ($(a1)+(d1)$) {}; {};
						\node[fvertexh] (8) at (3) {};
						\draw[gedge] (1)edge(2);
						\draw[gedge] (1)edge(3);
						\draw[gedge] (1)edge(4);
						\draw[gedge] (2)edge(3);
						\draw[gedge] (2)edge(8);
						\draw[bedge] (3)edge(5);
						\draw[redge] (3)edge(7);
						\draw[bedge] (4)edge(5);
						\draw[bedge] (5)edge(6);
						\draw[redge] (6)edge(7);
						\draw[redge] (7)edge(8);
					\end{scope}
				}
			\end{tikzpicture}
			&
			\begin{tikzpicture}[scale=0.8]
				\draw[sym] (0,-0.2) -- (0,2.3);
				\foreach \w/\o in {65/0.15,40/1}
				{
					\begin{scope}[opacity=\o]
						\coordinate (a1) at (\w:1);
						\coordinate (p1) at ($(0,0)!(a1)!(0,1)$);
						\coordinate (sa1) at ($(p1)+(p1)-(a1)$);
						\coordinate (d1) at (0,0.75);
						\coordinate (dt1) at (0.5,0);
						\node[fvertex] (1) at ($(0,0)-(dt1)$) {};
						\node[fvertex] (2) at (dt1) {};
						\node[fvertexh] (3) at (d1) {};
						\node[fvertexh] (4) at (3) {};
						\node[fvertex] (5) at ($(sa1)+(d1)$) {}; {};
						\node[fvertex] (6) at ($(a1)+(sa1)+(d1)$) {}; {};
						\node[fvertex] (7) at ($(a1)+(d1)$) {}; {};
						\node[fvertexh] (8) at (3) {};
						\draw[gedge] (1)edge(2);
						\draw[gedge] (1)edge(3);
						\draw[gedge] (1)edge(4);
						\draw[gedge] (2)edge(3);
						\draw[gedge] (2)edge(8);
						\draw[bedge] (3)edge(5);
						\draw[redge] (3)edge(7);
						\draw[bedge] (4)edge(5);
						\draw[redge] (5)edge(6);
						\draw[bedge] (6)edge(7);
						\draw[redge] (7)edge(8);
					\end{scope}
				}
			\end{tikzpicture}
		\end{tabular}
  \end{center}
  \caption{Flexible realisations (lower row) for the RS-colourings from \Cref{fig:pseudoRScol}.
  The RS-colourings are recalled in the upper row.
  Square vertices indicate those
  which have two  monochromatic paths in different colours between them,
  forcing them to overlap.}
  \label{fig:gridexample}
\end{figure}
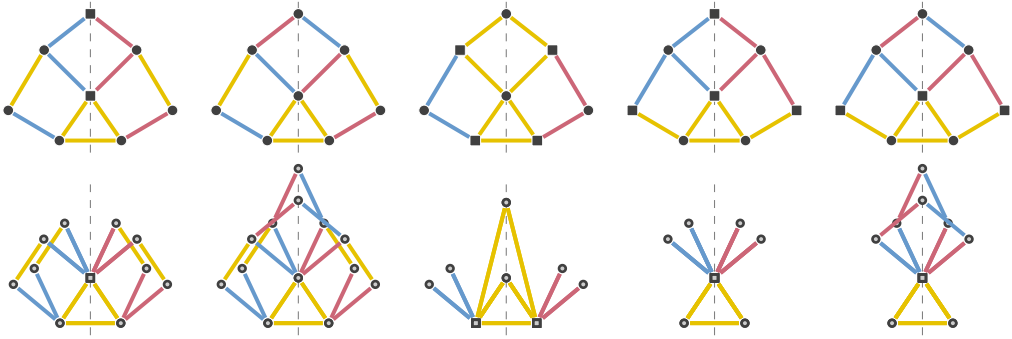

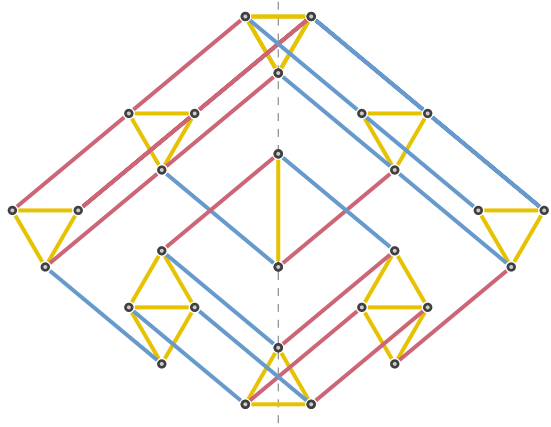
\begin{figure}[ht]
	\centering
	\begin{tikzpicture}[scale=2]
		\pgfmathparse{40}\let\w=\pgfmathresult
		\pgfmathparse{2}\let\lim=\pgfmathresult
		\draw[sym] (0,-0.25)--(0,2.6);
		\coordinate (d1) at (-30:0.25);
		\coordinate (d2) at (90:0.25);
		\coordinate (d3) at (210:0.25);
		\coordinate (m) at ($(d1)!(d2)!(d3)$);
		\coordinate (d4) at ($2*(m)-(d2)$);
		\foreach \d in {d1,d2,d3,d4}
		{
			\begin{scope}[shift=(\d)]
				\foreach \i in {0,1,2,3}
				{
					\foreach \j in {0,1,2,3}
					{
						\coordinate (d\d-\i-\j) at ($(\w:\i)+(180-\w:\j)$);
					}
				}
			\end{scope}
		}
		\foreach \d/\i/\j in {1/0/0,2/0/0,3/0/0,1/1/0,2/1/0,3/1/0,4/1/0,2/1/1,4/1/1,4/2/0,4/2/1,4/2/2,1/0/2,3/0/2,1/1/2,3/1/2,1/2/2,3/2/2}
		{
			\node[fvertex] (d\d-\i-\j) at (dd\d-\i-\j)  {};
			\node[fvertex] (d\d-\j-\i) at (dd\d-\j-\i)  {};
		}
		\draw[edge,colG] (d1-0-0)--(d2-0-0) (d1-0-0)--(d3-0-0) (d2-0-0)--(d3-0-0);
		\draw[edge,colG] (d1-1-0)--(d2-1-0) (d1-1-0)--(d3-1-0) (d1-1-0)--(d4-1-0) (d2-1-0)--(d3-1-0) (d3-1-0)--(d4-1-0);
		\draw[edge,colG] (d1-0-1)--(d2-0-1) (d1-0-1)--(d3-0-1) (d1-0-1)--(d4-0-1) (d2-0-1)--(d3-0-1) (d3-0-1)--(d4-0-1);
		\draw[edge,colG] (d2-1-1)--(d4-1-1);
		\draw[edge,colG] (d1-0-2)--(d3-0-2) (d1-0-2)--(d4-0-2) (d3-0-2)--(d4-0-2);
		\draw[edge,colG] (d1-2-0)--(d3-2-0) (d1-2-0)--(d4-2-0) (d3-2-0)--(d4-2-0);
		\draw[edge,colG] (d1-1-2)--(d3-1-2) (d1-1-2)--(d4-1-2) (d3-1-2)--(d4-1-2);
		\draw[edge,colG] (d1-2-1)--(d3-2-1) (d1-2-1)--(d4-2-1) (d3-2-1)--(d4-2-1);
		\draw[edge,colG] (d1-2-2)--(d3-2-2) (d1-2-2)--(d4-2-2) (d3-2-2)--(d4-2-2);
		\draw[edge,colG] (d1-2-2)--(d3-2-2) (d1-2-2)--(d4-2-2) (d3-2-2)--(d4-2-2);

		\draw[edge,colR] (d1-0-0)--(d1-1-0) (d2-0-0)--(d2-1-0) (d3-0-0)--(d3-1-0);
		\draw[edge,colB] (d1-0-0)--(d1-0-1) (d2-0-0)--(d2-0-1) (d3-0-0)--(d3-0-1);
		\draw[edge,colR] (d4-1-0)--(d4-2-0);
		\draw[edge,colB] (d4-0-1)--(d4-0-2);
		\draw[edge,colR] (d4-1-1)--(d4-2-1);
		\draw[edge,colB] (d4-1-1)--(d4-1-2);
		\draw[edge,colR] (d2-0-1)--(d2-1-1);
		\draw[edge,colB] (d2-1-0)--(d2-1-1);
		\draw[edge,colR] (d4-0-2)--(d4-1-2) (d4-1-2)--(d4-2-2);
		\draw[edge,colB] (d4-2-0)--(d4-2-1) (d4-2-1)--(d4-2-2);

		\draw[edge,colR] (d1-0-2)--(d1-1-2) (d3-0-2)--(d3-1-2) (d1-0-2)--(d1-1-2) (d1-1-2)--(d1-2-2) (d3-1-2)--(d3-2-2) (d1-1-2)--(d1-2-2);
		\draw[edge,colB] (d1-2-0)--(d1-2-1) (d3-2-0)--(d3-2-1) (d1-2-0)--(d1-2-1) (d1-2-1)--(d1-2-2) (d3-2-1)--(d3-2-2) (d1-2-1)--(d1-2-2);
	\end{tikzpicture}
	\caption{An example of a reflection-symmetric realisation obtained by the construction.}
	\label{fig:gridexample2}
\end{figure}

If a graph has no RS-colouring without an almost red-blue cycle,
we cannot apply \Cref{thm:noAlmostRBcycle2flex}.
For some graphs we can obtain a flex by combining the information of two RS-colourings; see \Cref{fig:RScolourings} for an example of this situation.
\begin{theorem}
	\label{thm:twoRS2flex}
	Let $G$ be a reflection-symmetric graph with reflection $\sigma$ and an invariant vertex $\tw$.
	Let $\delta_1,\delta_2$ be RS-colourings of $G$ that are certificates of each other
	for all almost red-blue cycles
	and $\delta$ be the map given
	by $\delta(e)=(\delta_1(e),\delta_2(e))$ for all $e\in E_G$.
	Also suppose:
	\begin{enumerate}
	  \item\label{it:twoRS2flex:sameGold} $\delta_1$ and $\delta_2$ have the same set of gold edges
	  		(hence, $|\delta(E_G)|\leq 5$),
	  \item there exist an invariant 5-cycle which is almost red-blue in both $\delta_1$ and $\delta_2$,
	  \item\label{it:twoRS2flex:neigh} there exists a unique partition $N,\sigma N$
	  		of the neighbours of $\tilde{w}$
	  		such that every two vertices in $N$ are connected
	  		by a path avoiding both $\tw$ and the gold edges,
	  \item\label{it:twoRS2flex:pcol} for every edge $uv\in E$, every path from $u$ to $v$ that does not contain any edge $e$ with $\delta(e)=\delta(uv)$
	  		must contain edges of the form $w_1\tw,\tw w_2$, where $w_1\in N$ and $w_2\in \sigma N$,
	  \item\label{it:twoRS2flex:allCombinations} for every cycle containing a path $(w_1,\tw,w_2)$ with $w_1 \in N$ and $w_2 \in \sigma N$, all possible five values of $\delta$ are attained over the cycle.   
	\end{enumerate}
	Then there is a reflection-symmetric flexible framework $(G,p)$.
\end{theorem}
\begin{proof}
	Let $H$ be the graph obtained from~$G$ by deleting $\tw$.
	We define the following graph by attaching two new vertices $\tw_1,\tw_2$ to $H$: 
	\begin{equation*}
		\tilde{G} = (V_H\cup\{\tw_1,\tw_2\}, E_H \cup \{u\tw_1 \colon u \in N\}\cup \{u\tw_2 \colon u \in \sigma N\})\,.
	\end{equation*}
	We extend the reflection of $H$ inherited from $G$ to a reflection $\sigma$ of $\tilde{G}$
	by setting $\sigma\tw_1=\tw_2$.
	In other words, we get $\tilde{G}$ by splitting the vertex $\tw$ from $G$
	and distributing its neighbours according to the partition.
	
	For each $\delta_i$, each almost red-blue
	cycle of $G$ passes from a vertex in $N$ to $\tw$ and then onto a vertex in $\sigma N$ by \cref{it:twoRS2flex:pcol}.
	Hence, $H$ has no almost red-blue cycles (with respect to each $\delta_i|_{E_H}$).
	For each~$\delta_i$, we get a pseudo-RS-colouring $\tdelta_i$ of $\tilde{G}$ by setting $\tdelta_i|_{E_H}=\delta_i|_{E_H}$ and $\tdelta_i(u\tw_j)=\delta_i(u\tw)$.
	Let $\tdelta(e)=(\tdelta_1(e),\tdelta_2(e))$ for all $e\in E_{\tilde{G}}$.
	Assume for contradiction there is an almost red-blue cycle in $(\tilde{G},\tdelta_i)$.
	Then it has to contain $\tw_1$ or $\tw_2$, otherwise it is an almost red-blue cycle of $H$,
	which is not possible by the fact above. 
	Assume it is $\tw_1$.
	Then the cycle contains the path $(u,\tw_1,v)$ with $u,v\in N$.
	But this path can be replaced by a red-blue path in $H$ from $u$ to $v$
	that exists by \cref{it:twoRS2flex:neigh}.
	This again gives an almost red-blue cycle in $H$ which is not possible.
	
	The goal is to construct a reflection-symmetric flex $p_t$ of $\tilde{G}$ such that
	$p_t(\tw_1)=p_t(\tw_2)$, since then setting $p_t(\tw)=p_t(\tw_1)$ gives a reflection-symmetric flex of $G$.
	In order to do so, let $a,b:V_{\tilde{G}}\rightarrow \RR^2$ be maps
	such that $a(u)=a(v)$ if and only if there is a path from $u$ to $v$ in $\tilde{G}$
	which has no edge $e$ with $\tdelta(e)=(\blue,\blue)$,
	while $b(u)=b(v)$ if and only if there is a path from $u$ to $v$ in $\tilde{G}$
	which has no edge $e$ with $\tdelta(e)=(\blue,\red)$.
	Let $z:V_{\tilde{G}}\rightarrow \RR^2$ be the same as in the proof of \Cref{thm:noAlmostRBcycle2flex}.
	We note that both $\tdelta_1$ and $\tdelta_2$ give the same $z$ (given the fixed $d_i$'s and $\bar{d}_j$'s) as they have the same gold edges by \ref{it:twoRS2flex:sameGold}.
	For $s,t \in [0,2\pi]$ we define a realisation $p_{s,t}$ of $\tilde{G}$ given by
	\begin{equation*}
		p_{s,t}(u) :=\rott a(u)	+ \rotmt \tau a(\sigma u) +\rots b(u)	+ \rotms \tau b(\sigma u) + z(u) \, .
	\end{equation*}
	Analogously to the proof of \Cref{thm:noAlmostRBcycle2flex}, $p_{s,t}(\sigma u)=\tau p_{s,t}(u)$.
	We have to check that for every edge $uv\in E_{\tilde{G}}$, the distance $\|p_{s,t}(u)-p_{s,t}(v)\|$
	is independent of $s$ and $t$.
	If $\tdelta(uv)=(\blue,\blue)$, then $\tdelta(\sigma(uv))=(\red,\red)$,
	therefore $a(\sigma u)=a(\sigma v)$.
	It also holds that $b(u)=b(v)$, $b(\sigma u)=b(\sigma v)$ and $z(u)=z(v)$.
	This gives
	\begin{equation}
		\label{eq:blueblue}
		\|p_{s,t}(u)-p_{s,t}(v)\| = \|\rott (a(u)-a(v))\| = \|a(u)-a(v)\|\,. 
	\end{equation} 
	In case $\tdelta(uv)=(\blue,\red)$, we obtain
	$\|p_{s,t}(u)-p_{s,t}(v)\| = \|\rots (b(u)-b(v))\| = \|b(u)-b(v)\|$.
	If $\tdelta(uv)=(\gold,\gold)$, we have
	\begin{equation}
		\label{eq:goldgold}
		\|p_{s,t}(u)-p_{s,t}(v)\| = \|z(u)-z(v)\|\,. 
	\end{equation}
	The remaining two combinations of colours follow by symmetry.
	
	Now we show that $p_{s,t}(u)\neq p_{s,t}(v)$ for every edge $uv$.
	Suppose for contradiction that $p_{s,t}(u)= p_{s,t}(v)$ for some edge $uv$.
	If $uv$ is gold, then $z(u)=z(v)$ by \eqref{eq:goldgold}.
	But this means that $u$ and $v$ are connected by a red-blue path,
	which contradicts that there is no almost red-blue cycle in $\tdelta_i$.
	If $\tdelta(uv)=(\blue,\blue)$, then $a(u)=a(v)$ by \eqref{eq:blueblue}.
	Hence, there is a path from $u$ to $v$ such that $\tdelta(e)\neq (\blue,\blue)$
	for every edge $e$ in the path.
	The path must contain $\tw_1$ or $\tw_2$,
	otherwise the corresponding path in $G$ violates \cref{it:twoRS2flex:pcol}.
	Suppose that the path contains both $\tw_1$ and $\tw_2$.
	Then the part from $\tw_1$ to $\tw_2$ forms a cycle in $G$
	that contradicts \cref{it:twoRS2flex:allCombinations} as it does not contain $(\blue,\blue)$.
	If only $\tw_1$ is contained, then we get a contradiction with \cref{it:twoRS2flex:pcol}
	since both the predecessor and successor of $\tw_1$ in the path are in $N$.
	The situation is analogous for all remaining combinations of $\tdelta(uv)$.
	
	We are left to set $s$ as a function of $t$ so that $p_{s(t),t}(\tw_1)$
	is on the $y$-axis,
	since then $p_{s(t),t}(\tw_1)=p_{s(t),t}(\tw_2)$ by the symmetry
	and $p_t = p_{s(t),t}$ is the desired reflection-symmetric flex.
	Let $\bar{u}(\sigma\bar{u})$ be the gold edge of the invariant 5-cycle
	and $x$ be the vertex adjacent to $\bar{u}$ and~$\tw_1$.
	Notice that \cref{it:twoRS2flex:pcol,it:twoRS2flex:allCombinations} imply that $\delta_1$ and $\delta_2$
	are certificates for each other for all almost red-blue cycles. 
	By swapping $\tdelta_1$ with $\tdelta_2$ or taking conjugates,
	we can assume that $\tdelta(\bar{u}x)=(\blue,\blue)$
	and $\tdelta(x\tw_1)=(\blue,\red)$.
	Since $(\tw_2,\sigma x,\sigma \bar{u},\bar{u})$ is a path avoiding $(\blue,\blue)$,
	we have $a(\tw_2)=a(\sigma x)=a(\sigma \bar{u})=a(\bar{u})$,
	which we can assume to be $(0,0)$.
	Similarly, we have $b(\tw_2)=b(\sigma x)=b(\sigma \bar{u})=b(\bar{u})=b(x)$,
	which we again assume to be $(0,0)$.
	Hence, $p_{s,t}(\bar{u})=z(\bar{u})$ does not lie on the $y$-axis
	since the edge $\bar{u}\sigma\bar{u}$ has non-zero length.
	Therefore, we can set $z(\bar{u}):=(1,0)$.
	With this, $p_{s,t}(x)$ and $p_{s,t}(\tw_1)$ are as indicated in \Cref{fig:param}.
	We can choose $a(x)$ so that the length of the edge $\bar{u}x$ is $\|a(x)\|=2$, 
	and $b\tw_1$ so that the length of $x\tw_1$ is $\|b(\tw_1)\|=2$.
	Then for every $t$, there is one or two values of $s$ such that $p_{s,t}(\tw_1)$
	is on the $y$-axis. To get the desired function $s(t)$,
	we pick the values for $s(t)$ such that $p_{s(t),t}(\tw_1)$ is continuous.
\end{proof}

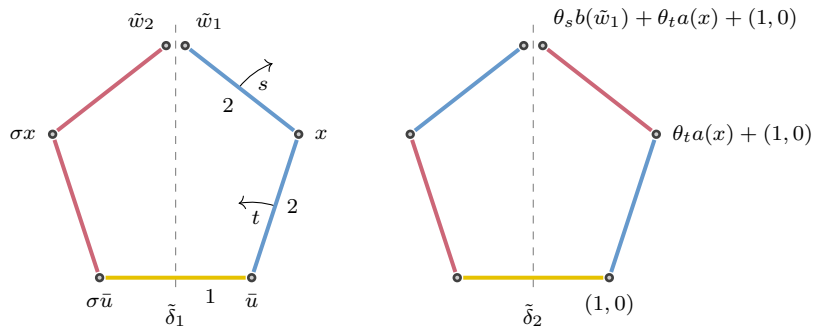
\begin{figure}[ht]
	\centering
	\begin{tikzpicture}[scale=0.5]
		\draw[sym] (0,-0.25)--(0,6.75) ;
		\coordinate (c) at (0.25,0);
		\node[fvertex,label={[labelsty]-90:$\sigma \bar u$}] (1l) at (-2,0) {};
		\node[fvertex,label={[labelsty]-90:$\bar u$}] (1r) at (2,0) {};
		\node[fvertex,label={[labelsty]180:$\sigma x$}] (2l) at ($(1l)+(108:4)$) {};
		\node[fvertex,label={[labelsty]0:$x$}] (2r) at ($(1r)+(72:4)$) {};
		\coordinate (w) at ($(2r)+(144:4)$);
		\node[fvertex,label={[labelsty]85:$\tilde w_1$}] (wr) at ($(w)+(c)$) {};
		\node[fvertex,label={[labelsty]95:$\tilde w_2$}] (wl) at ($(w)-(c)$) {};

		\coordinate (i1) at ($(1r)+(100:1)$);
		\draw[black] pic (wt) ["",draw,->,angle radius=1cm, angle eccentricity=0.8] {angle= 2r--1r--i1};
		\node[labelsty] at (wt) {$t$};
		\coordinate (i2) at ($(2r)+(110:1)$);
		\draw[black] pic (ws) ["",draw,<-,angle radius=1cm, angle eccentricity=0.8] {angle= i2--2r--wr};
		\node[labelsty] at (ws) {$s$};

		\draw[edge,colG] (1l) to node[below,colfg,labelsty,pos=0.75] {1} (1r);
		\draw[edge,colR] (1l)--(2l) (2r) (2l)--(wl);
		\draw[edge,colB] (1r) to node[right,colfg,labelsty] {2} (2r) to node[below left=-2pt,colfg,labelsty] {2} (wr);
		\node[labelsty] at (0,-1) {$\tilde \delta_1$};
	\end{tikzpicture}
	\qquad
	\begin{tikzpicture}[scale=0.5]
		\draw[sym] (0,-0.25)--(0,6.75);
		\coordinate (c) at (0.25,0);
		\node[fvertex] (1l) at (-2,0) {};
		\node[fvertex,label={[labelsty]-90:$(1,0)$}] (1r) at (2,0) {};
		\node[fvertex] (2l) at ($(1l)+(108:4)$) {};
		\node[fvertex,label={[labelsty]0:$\theta_t a(x) +(1,0)$}] (2r) at ($(1r)+(72:4)$) {};
		\coordinate (w) at ($(2r)+(144:4)$);
		\node[fvertex,label={[labelsty]85:$\theta_s b(\tilde w_1) + \theta_t a(x) + (1,0) $}] (wr) at ($(w)+(c)$) {};
		\node[fvertex] (wl) at ($(w)-(c)$) {};

		\draw[edge,colG] (1l)--(1r);
		\draw[edge,colB] (1r)--(2r) (2l)--(wl);
		\draw[edge,colR] (1l)--(2l) (2r)--(wr);
		\node[labelsty] at (0,-1) {$\tilde \delta_2$};
	\end{tikzpicture}
	\caption{The RS-colouring $\tilde \delta_1$ with vertex names and chosen lengths (left) and its certificate $\tilde \delta_2$ with parametrisations of the vertex coordinates (right).}
	\label{fig:param}
\end{figure}

\Cref{fig:colourcomb} shows that the additional conditions on $\delta_1$ and $\delta_2$
are required for the proof.

\begin{figure}[ht]
	\centering
	\begin{tikzpicture}[scale=0.8]
		\draw[sym] (0,-0.25)--(0,6.5);
		\node[gvertex] (1l) at (-2,0) {};
		\node[gvertex] (1r) at (2,0) {};
		\node[gvertex] (2l) at ($(1l)+(108:4)$) {};
		\node[gvertex] (2r) at ($(1r)+(72:4)$) {};
		\node[gvertex] (w) at ($(2r)+(144:4)$) {};

		\node[gvertex] (3l) at ($(w)+(-110:3)$) {};
		\node[gvertex] (3r) at ($(w)+(-70:3)$) {};

		\node[gvertex] (4l) at ($(3l)+(200:0.75)$) {};
		\node[gvertex] (4r) at ($(3r)+(-20:0.75)$) {};
		\node[gvertex] (5l) at ($(4l)+(180:0.5)$) {};
		\node[gvertex] (5r) at ($(4r)+(0:0.5)$) {};
		\node[gvertex] (6l) at ($(5l)+(140:0.5)$) {};
		\node[gvertex] (6r) at ($(5r)+(40:0.5)$) {};

		\node[gvertex] (8l) at (-1,1.2) {};
		\node[gvertex] (8r) at (1,1.2) {};
		\node[gvertex] (7l) at ($(8l)+(140:1.2)$) {};
		\node[gvertex] (7r) at ($(8r)+(40:1.2)$) {};
		\node[gvertex] (9l) at ($(8l)+(30:1.8)$) {};
		\node[gvertex] (9r) at ($(8r)+(150:1.8)$) {};
		\node[gvertex] (10l) at ($(9l)+(30:0.75)$) {};
		\node[gvertex] (10r) at ($(9r)+(150:0.75)$) {};

		\draw[edge,colG] (1l)--(1r);
		\draw[edge,colR] (1l)--(2l) (2r)--(w);
		\draw[edge,colB] (1r)--(2r) (2l)--(w);

		\draw[edge,colR] (3r)--(w) (3r)--(4r) (6r)--(2r) (3r)--(7r) (7r)--(8r) (3r)--(10l) (9l)--(10l) (4l)--(5l) (5l)--(6l);
		\draw[edge,colB] (3l)--(w) (3l)--(4l) (6l)--(2l) (3l)--(7l) (7l)--(8l) (3l)--(10r) (9r)--(10r) (4r)--(5r) (5r)--(6r);
		\draw[edge,colG] (8r)--(9r);
		\draw[edge,colG] (8l)--(9l);
	\end{tikzpicture}
	\qquad
	\begin{tikzpicture}[scale=0.8]
		\draw[sym] (0,-0.25)--(0,6.5);
		\node[gvertex] (1l) at (-2,0) {};
		\node[gvertex] (1r) at (2,0) {};
		\node[gvertex] (2l) at ($(1l)+(108:4)$) {};
		\node[gvertex] (2r) at ($(1r)+(72:4)$) {};
		\node[gvertex] (w) at ($(2r)+(144:4)$) {};

		\node[gvertex] (3l) at ($(w)+(-110:3)$) {};
		\node[gvertex] (3r) at ($(w)+(-70:3)$) {};

		\node[gvertex] (4l) at ($(3l)+(200:0.75)$) {};
		\node[gvertex] (4r) at ($(3r)+(-20:0.75)$) {};
		\node[gvertex] (5l) at ($(4l)+(180:0.5)$) {};
		\node[gvertex] (5r) at ($(4r)+(0:0.5)$) {};
		\node[gvertex] (6l) at ($(5l)+(140:0.5)$) {};
		\node[gvertex] (6r) at ($(5r)+(40:0.5)$) {};

		\node[gvertex] (8l) at (-1,1.2) {};
		\node[gvertex] (8r) at (1,1.2) {};
		\node[gvertex] (7l) at ($(8l)+(140:1.2)$) {};
		\node[gvertex] (7r) at ($(8r)+(40:1.2)$) {};
		\node[gvertex] (9l) at ($(8l)+(30:1.8)$) {};
		\node[gvertex] (9r) at ($(8r)+(150:1.8)$) {};
		\node[gvertex] (10l) at ($(9l)+(30:0.75)$) {};
		\node[gvertex] (10r) at ($(9r)+(150:0.75)$) {};

		\draw[edge,colG] (1l)--(1r);
		\draw[edge,colR] (1r)--(2r) (2r)--(w);
		\draw[edge,colB] (1l)--(2l) (2l)--(w);

		\draw[edge,colR] (3r)--(w) (3l)--(4l) (4r)--(5r) (5r)--(6r) (6l)--(2l) (3l)--(7l) (7r)--(8r) (3r)--(10l) (9l)--(10l) ;
		\draw[edge,colB] (3l)--(w) (3r)--(4r) (4l)--(5l) (5l)--(6l) (6r)--(2r) (3r)--(7r) (7l)--(8l) (3l)--(10r) (9r)--(10r) ;
		\draw[edge,colG] (8r)--(9r);
		\draw[edge,colG] (8l)--(9l);
	\end{tikzpicture}
	\caption{An example where condition~\cref{it:twoRS2flex:allCombinations} does not hold.
	The construction of a flex in the proof of \Cref{thm:twoRS2flex}
	gives the edges incident to the invariant edge zero length, which is not allowed.}
	\label{fig:colourcomb}
\end{figure}

\section{Frameworks consisting of triangles and parallelograms}\label{sec:ptp}
Frameworks formed by triangles and parallelograms are studied in~\cite{Bracing,TPframe}.
Their flexibility is characterised using special types of NAC-colourings.
In this section we show how the flexibility of the reflection-symmetric ones relates to a special type of RS-colourings.
We start by recalling the defining properties of the frameworks under consideration. 
\begin{definition}
    \label{def:APclass}
	Let $G$ be a connected graph. Consider the relation $\Trel$ on the set of edges, where
	two edges are in relation $\Trel$ if they are in a 3-cycle subgraph of $G$.
	Two edges are in relation $\Prel$ if they are opposite edges of a 4-cycle subgraph of $G$.
	An equivalence class of the reflexive-transitive closure
	of the union $\Trel \cup \Prel$ is called an \emph{angle-preserving class}.
\end{definition}
For a walk $W=(u_1,\ldots,u_k)$ and an angle-preserving class $r$,
the notation $\sum_{\oriented{u}{v} \in r\cap W}$ means
we sum up over all edges $(u_i,u_{i+1})$ with $1\leq i <k$ such that $u_i u_{i+1}\in r$.
\begin{definition}
    We say that a framework $(G,p)$ is \emph{walk-independent} if $p$ is an injective realisation where each induced 4-cycle in $G$ forms a non-degenerate\footnote{Here by non-degenerate we mean that not all of its vertices are collinear.}
	parallelogram in $p$, and for every angle-preserving class $r$ we have
    \begin{equation*}
		\sum_{\oriented{u}{v} \in r\cap W} (p(v)-p(u)) = \sum_{\oriented{u}{v} \in r\cap W'} (p(v)-p(u))
	\end{equation*}
	for every $w_1,w_2\in V_G$  and for all walks $W,W'$ in $G$ from $w_1$ to $w_2$.
\end{definition}
\Cref{fig:walkindep} illustrates the definition of walk-independence by showing the sum of vectors of an angle-preserving class for a given walk.

\begin{figure}[ht]
	\centering
	\begin{tikzpicture}[scale=0.8]
		\pgfmathparse{70}\let\w=\pgfmathresult
		\begin{scope}
			\draw[sym] (0,-0.2) -- (0,2.3);
			\coordinate (a1) at (\w:1.25);
			\coordinate (p1) at ($(0,0)!(a1)!(0,1)$);
			\coordinate (sa1) at ($(p1)+(p1)-(a1)$);
			\coordinate (d1) at (0,0.75);
			\coordinate (dt1) at (0.5,0);
			\node[fvertex] (1) at ($(0,0)-(dt1)$) {};
			\node[fvertex] (2) at (dt1) {};
			\node[fvertex] (3) at (d1) {};
			\node[fvertex] (4) at ($(sa1)-(dt1)$) {};
			\node[fvertex] (5) at ($(sa1)+(d1)$) {}; {};
			\node[fvertex] (6) at ($(a1)+(sa1)+(d1)$) {}; {};
			\node[fvertex] (7) at ($(a1)+(d1)$) {}; {};
			\node[fvertex] (8) at ($(a1)+(dt1)$) {};
			\draw[hedge2,colG] (1)edge(2);
			\draw[hedge2,colG] (1)edge(3);
			\draw[ledge,colB] (1)edge(4);
			\draw[ledge,colG] (2)edge(3);
			\draw[hedge2,colR] (2)edge(8);
			\draw[hedge2,colB] (3)edge(5);
			\draw[ledge,colR] (3)edge(7);
			\draw[ledge,colG] (4)edge(5);
			\draw[hedge2,colR] (5)edge(6);
			\draw[hedge2,colB] (6)edge(7);
			\draw[hedge2,colG] (7)edge(8);
		\end{scope}
		\begin{scope}[xshift=5cm]
			\draw[sym] (0,-0.2) -- (0,2.3);
			\coordinate (a1) at (\w:1.25);
			\coordinate (p1) at ($(0,0)!(a1)!(0,1)$);
			\coordinate (sa1) at ($(p1)+(p1)-(a1)$);
			\coordinate (d1) at (0,0.75);
			\coordinate (dt1) at (0.5,0);
			\node[fvertex] (1) at ($(0,0)-(dt1)$) {};
			\node[fvertex] (2) at (dt1) {};
			\node[fvertex] (3) at (d1) {};
			\node[fvertex] (4) at ($(sa1)-(dt1)$) {};
			\node[fvertex] (5) at ($(sa1)+(d1)$) {}; {};
			\node[fvertex] (6) at ($(a1)+(sa1)+(d1)$) {}; {};
			\node[fvertex] (7) at ($(a1)+(d1)$) {}; {};
			\node[fvertex] (8) at ($(a1)+(dt1)$) {};
			\draw[dedge,colG] (2)edge(1);
			\draw[dedge,colG] (1)edge(3);
			\draw[dedge,colR] (8)edge(2);
			\draw[dedge,colB] (3)edge(5);
			\draw[dedge,colR] (5)edge(6);
			\draw[dedge,colB] (6)edge(7);
			\draw[dedge,colG] (7)edge(8);
		\end{scope}
		\begin{scope}[xshift=10cm,yshift=1cm]
			\coordinate (a1) at (\w:1.25);
			\coordinate (p1) at ($(0,0)!(a1)!(0,1)$);
			\coordinate (sa1) at ($(p1)+(p1)-(a1)$);
			\coordinate (d1) at (0,0.75);
			\coordinate (dt1) at (0.5,0);
			\coordinate (1) at ($(0,0)-(dt1)$);
			\coordinate (2) at (dt1);
			\coordinate (3) at (d1);
			\coordinate (4) at ($(sa1)-(dt1)$);
			\coordinate (5) at ($(sa1)+(d1)$);;
			\coordinate (6) at ($(a1)+(sa1)+(d1)$);;
			\coordinate (7) at ($(a1)+(d1)$);;
			\coordinate (8) at ($(a1)+(dt1)$);
			\draw[axes] (0,-1.3)--(0,1.3);
			\draw[axes] (-1.25,0)--(1,0);
			\draw[dedge,colB] (0,0)--($(5)-(3)$);
			\draw[dedge,colB] ($(5)-(3)$)--++($(7)-(6)$);
			\draw[dedge,colG] (0,0)--($(1)-(2)$);
			\draw[dedge,colG] ($(1)-(2)$)--++($(3)-(1)$);
			\draw[dedge,colG] ($(1)-(2)$)++($(3)-(1)$)--++($(8)-(7)$);
			\draw[dedge,colR] (0,0)--($(2)-(8)$);
			\draw[dedge,colR] ($(2)-(8)$)--++($(6)-(5)$);
		\end{scope}
	\end{tikzpicture}
	\caption{A reflection-symmetric framework with a closed walk indicated in bold (left); the directed walk with colours of angle-preserving classes (middle); the vectors of the edges from the walk for each angle-preserving class (right). One can see that they sum up to zero in this example.}
	\label{fig:walkindep}
\end{figure}
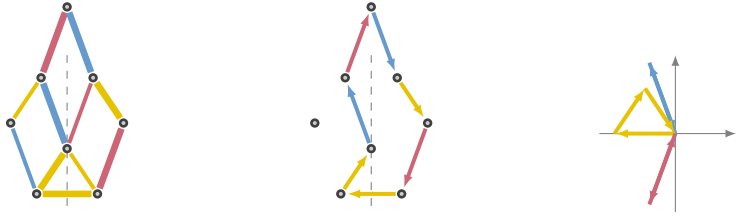

The type of RS-colourings needed in this section require extra conditions on path colours.

\begin{definition}
	Let $G$ be a reflection-symmetric graph.
	A pseudo-RS-colouring is called \emph{Cartesian}
	if there exists no distinct pair of vertices $v,w \in V_G$ and triple of paths $P_{\mathrm{rb}},P_{\mathrm{rg}},P_{\mathrm{bg}}$  where the following holds: (i) each path starts at $v$ and ends at $w$, (ii) the edges of $P_{\mathrm{rb}}$ are red or blue, (iii) the
	edges of $P_{\mathrm{rg}}$ are red or gold, and (iv) edges of $P_{\mathrm{bg}}$ are blue or gold.
\end{definition}
Notice that in the definition, a red path would count for both $P_{\mathrm{rb}}$ and $P_{\mathrm{rg}}$.
From the RS-colourings in \Cref{fig:pseudoRScol}, only the second one is Cartesian.
\Cref{fig:cartesian} illustrates some examples and non-examples of Cartesian RS-colourings.

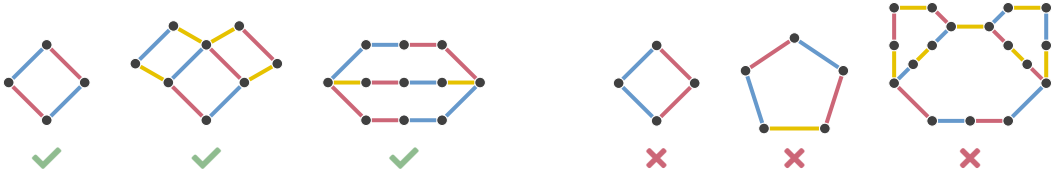
\begin{figure}[ht]
		\centering
		\begin{tikzpicture}
			\node[gvertex] (1) at (180:0.5) {};
			\node[gvertex] (2) at (-90:0.5) {};
			\node[gvertex] (3) at (0:0.5) {};
			\node[gvertex] (4) at (90:0.5) {};
			\draw[edge,colR] (1)--(2) (3)--(4);
			\draw[edge,colB] (2)--(3) (4)--(1);
			\node[colGr] at ($(2)+(0,-0.5)$) {\faCheck};
		\end{tikzpicture}
		\quad
		\begin{tikzpicture}
			\node[gvertex] (1) at (180:0.5) {};
			\node[gvertex] (2) at (-90:0.5) {};
			\node[gvertex] (3) at (0:0.5) {};
			\node[gvertex] (4) at (90:0.5) {};
			\node[gvertex] (5) at ($(3)+(30:0.5)$) {};
			\node[gvertex] (6) at ($(4)+(30:0.5)$) {};
			\node[gvertex] (7) at ($(1)+(150:0.5)$) {};
			\node[gvertex] (8) at ($(4)+(150:0.5)$) {};
			\draw[edge,colR] (1)--(2) (3)--(4) (5)--(6);
			\draw[edge,colB] (2)--(3) (4)--(1) (7)--(8);
			\draw[edge,colG] (3)--(5) (4)--(6) (1)--(7) (4)--(8);
			\node[colGr] at ($(2)+(0,-0.5)$) {\faCheck};
		\end{tikzpicture}
		\quad
		\begin{tikzpicture}
			\node[gvertex] (1) at (0,0) {};
			\node[gvertex] (2a) at (0.5,-0.5) {};
			\node[gvertex] (2b) at (0.5,0) {};
			\node[gvertex] (2c) at (0.5,0.5) {};
			\node[gvertex] (3a) at (1,-0.5) {};
			\node[gvertex] (3b) at (1,0) {};
			\node[gvertex] (3c) at (1,0.5) {};
			\node[gvertex] (4a) at (1.5,-0.5) {};
			\node[gvertex] (4b) at (1.5,0) {};
			\node[gvertex] (4c) at (1.5,0.5) {};
			\node[gvertex] (5) at (2,0) {};
			\draw[edge,colR] (1)--(2a) (2a)--(3a) (2b)--(3b) (3c)--(4c) (4c)--(5);
			\draw[edge,colB] (3a)--(4a) (4a)--(5) (3b)--(4b) (1)--(2c) (2c)--(3c);
			\draw[edge,colG] (1)--(2b) (4b)--(5);
			\node[colGr] at ($(3a)+(0,-0.5)$) {\faCheck};
		\end{tikzpicture}
		\qquad\qquad
		\begin{tikzpicture}
			\node[gvertex] (1) at (-90:0.5) {};
			\node[gvertex] (2) at (0:0.5) {};
			\node[gvertex] (3) at (90:0.5) {};
			\node[gvertex] (4) at (180:0.5) {};
			\draw[edge,colR] (1)--(2) (2)--(3);
			\draw[edge,colB] (3)--(4) (4)--(1);
			\node[colR] at ($(1)+(0,-0.5)$) {\faTimes};
		\end{tikzpicture}
		\quad
		\begin{tikzpicture}[scale=0.8]
			\node[gvertex] (1) at (-0.5,0) {};
			\node[gvertex] (2) at (0.5,0) {};
			\node[gvertex] (3) at (50:1.25) {};
			\node[gvertex] (4) at (130:1.25) {};
			\node[gvertex] (5) at (0,1.5) {};
			\draw[edge,colR] (2)--(3) (5)--(4);
			\draw[edge,colB] (3)--(5) (4)--(1);
			\draw[edge,colG] (1)--(2);
			\node[colR] at (0,-0.5) {\faTimes};
		\end{tikzpicture}
		\quad
		\begin{tikzpicture}
			\node[gvertex] (3bl) at (-1,0) {};
			\node[gvertex] (2bl) at (-0.5,-0.5) {};
			\node[gvertex] (1b) at (0,-0.5) {};
			\node[gvertex] (2br) at (0.5,-0.5) {};
			\node[gvertex] (3br) at (1,0) {};

			\node[gvertex] (4tl) at ($(3bl)+(0,0.5)$) {};
			\node[gvertex] (3tl) at ($(4tl)+(0,0.5)$) {};
			\node[gvertex] (2tl) at ($(3tl)+(0.5,0)$) {};
			\node[gvertex] (1tl) at (-0.25,0.75) {};

			\node[gvertex] (4tr) at ($(3br)+(0,0.5)$) {};
			\node[gvertex] (3tr) at ($(4tr)+(0,0.5)$) {};
			\node[gvertex] (2tr) at ($(3tr)+(-0.5,0)$) {};
			\node[gvertex] (1tr) at (0.25,0.75) {};

			\node[gvertex] (4ml) at (-0.75,0.25) {};
			\node[gvertex] (3ml) at (-0.5,0.5) {};
			\node[gvertex] (3mr) at (0.5,0.5) {};
			\node[gvertex] (4mr) at (0.75,0.25) {};

			\draw[edge,colR] (3bl)--(2bl) (1b)--(2br);
			\draw[edge,colB] (2bl)--(1b) (2br)--(3br);
			\draw[edge,colG] (1tl)--(1tr);

			\draw[edge,colR] (1tl)--(2tl) (3tl)--(4tl);
			\draw[edge,colB] (1tr)--(2tr) (3tr)--(4tr);
			\draw[edge,colG] (2tl)--(3tl) (4tl)--(3bl) (2tr)--(3tr) (4tr)--(3br);

			\draw[edge,colG] (3ml)--(4ml);
			\draw[edge,colG] (3mr)--(4mr);
			\draw[edge,colB] (1tl)--(3ml) (4ml)--(3bl);
			\draw[edge,colR] (1tr)--(3mr) (4mr)--(3br);
			\node[colR] at ($(1b)+(0,-0.5)$) {\faTimes};
		\end{tikzpicture}
		\caption{Edge colourings that are and are not Cartesian RS-colourings.}
		\label{fig:cartesian}
	\end{figure}

It follows from the definition that Cartesian pseudo-RS-colourings do not have almost red-blue cycles,
and so all Cartesian pseudo-RS-colourings are RS-colourings.
One can show that the reflection-symmetric flex constructed in \Cref{thm:noAlmostRBcycle2flex}
has an injective realisation if and only if the used RS-colouring is Cartesian.

We now describe how Cartesian RS-colourings of walk-independent frameworks
relate to angle-preserving classes.

\begin{lemma}
	\label{lem:cartesianRSiffAPCmonochromatic}
	Let $\delta$ be an edge colouring by red, blue and gold
	of a reflection-symmetric walk-independent framework $(G,p)$ with reflection $\sigma$.
	Then $\delta$ is a Cartesian RS-colouring if and only if both red and blue colours occur,
	each invariant angle-preserving class is gold,
	and for every non-invariant angle-preserving class $r$
	either both $r$ and $\sigma r$ are gold, or one is blue and the other one is red.
\end{lemma}
\begin{proof}
	Suppose that $\delta$ is a Cartesian RS-colouring.
	We only need to show that the angle preserving classes are monochromatic;
	the rest then follows from the properties of pseudo-RS-colourings.
	Since 3-cycles are always monochromatic in an RS-colouring, edges in $\Trel$ relation have the same colour.
	Since at most two colours can occur in a 4-cycle and no two vertices are connected
	by two monochromatic paths of distinct colour simultaneously,
	every 4-cycle is either monochromatic or the opposite edges have the same colour.
	Hence, the angle preserving classes are monochromatic.
	
	For the opposite direction,
	clearly, \cref{it:containsRedBlue,it:symmetryRedBlue} in \Cref{def:pseudoRScol} hold.
	Since angle-preserving classes are edge cuts by~\cite[Lemma~1.5]{TPframe} and monochromatic by assumption,
	if a colour occurs in a cycle, it must occur at least twice.
	Therefore, changing all gold edges to red, resp.\ blue, yields a NAC-colouring.
	This also gives that there is no almost red-blue cycle,
	namely, $\delta$ is indeed an RS-colouring.
	By \cite[Lemma~1.5]{TPframe}, every two vertices are separated by an angle-preserving class,
	which is monochromatic. Hence, $\delta$ is Cartesian. 
\end{proof}
	
\begin{theorem}
	\label{thm:flexibleIffCartRScolouring}
	Let $(G,p)$ be a reflection-symmetric walk-independent framework.
	The following are equivalent:
	\begin{enumerate}
	  \item\label{it:RSflexible} $(G,p)$ is reflection-symmetric flexible,
	  \item\label{it:cartRS} $G$ has a Cartesian RS-colouring,
	  \item\label{it:noninvAPC} $G$ has a non-invariant angle-preserving class.
	\end{enumerate}
\end{theorem}
\begin{proof}
	\cref{it:noninvAPC}$\implies$\cref{it:cartRS}: Colour a non-invariant angle-preserving class red,
	its image under $\sigma$ blue and the rest gold.
	This gives a Cartesian RS-colouring 
	by \Cref{lem:cartesianRSiffAPCmonochromatic}.
	
	\cref{it:RSflexible}$\implies$\cref{it:noninvAPC}:
	By \cite[Proposition 1.2]{TPframe}, there are at least two angle-preserving classes.
	There must be an edge $e$ whose angle with the $y$-axis changes along the flex,
	otherwise the flex is trivial.
	Hence, the angle between $e$ and $\sigma e$ changes as well,
	thus, they cannot be in the same angle-preserving class (edges in $\Trel$ or $\Prel$ keep their mutual angles),
	namely, the angle-preserving class of $e$ is not invariant.

	\cref{it:cartRS}$\implies$\cref{it:RSflexible}:
	We use the construction from \Cref{thm:noAlmostRBcycle2flex},
	with points $a_i,d_i$ tailored to obtain a flex starting at $p$.
	Let $\delta$ be a Cartesian RS-colouring of $G$.
	We fix a vertex $\bar{u}\in V_G$ and assume that $p(\bar{u})$ is on the $x$-axis.
	For $v\in V_G$ and a path $W$ from $\bar{u}$ to $v$, we define 
	\begin{equation*}
		p_\red(v) := \hspace{-0.2cm}\sum_{\substack{\oriented{u_1}{u_2} \in W\\ \delta(u_1u_2)=\red}}\hspace{-0.2cm} (p(u_2)-p(u_1))\,.
	\end{equation*}
	As the angle-preserving classes are monochromatic (\Cref{lem:cartesianRSiffAPCmonochromatic})
	and $(G,p)$ is walk-independent, $p_\red(v)$ is independent of the choice of $W$.
	Analogously we define $p_\blue(v)$ and $p_\gold(v)$.
	
	We need to show a few identities.
	Let $\bar{W}$ be a path from $\bar{u}$ to $\sigma\bar{u}$
	and $W$ from $\bar{u}$ to~$v$. 
	Since $\sigma W$ is a path from $\sigma\bar{u}$ to $\sigma v$
	and $\bar{W}$ concatenated with $\sigma W$ is from $\bar{u}$ to~$\sigma v$ (illustrated in \Cref{fig:flexibleIffCartRScolouring:paths}),
	we have
	\begin{align}
		\label{eq:pbluesigmav}
		p_\blue(\sigma v) 
			&= \hspace{-0.2cm}\sum_{\substack{\oriented{u_1}{u_2} \in \bar{W}\\ \delta(u_1u_2)=\blue}}\hspace{-0.2cm} (p(u_2)-p(u_1))
				+ \hspace{-0.2cm}\sum_{\substack{\oriented{\sigma u_1}{\sigma u_2} \in \sigma W\\ \delta(\sigma u_1\sigma u_2)=\blue}}\hspace{-0.2cm} (p(\sigma u_2)-p(\sigma u_1)) \nonumber \\ 
			&= p_\blue(\sigma \bar{u})
				+ \hspace{-0.2cm}\sum_{\substack{\oriented{u_1}{u_2} \in W\\ \delta(u_1u_2)=\red}}\hspace{-0.2cm} \tau(p(u_2)-p(u_1))
			 = p_\blue(\sigma \bar{u}) + \tau p_\red(v)\,.
	\end{align}

	\begin{figure}[ht]
		\centering
		\begin{tikzpicture}[ifvertex/.style={fvertex,minimum size=5pt}]
			\node[ifvertex,label={[labelsty]0:$\bar u$}] (u) at (2,0) {};
			\node[ifvertex,label={[labelsty]180:$\sigma \bar u$}] (su) at (-2,0) {};
			\node[ifvertex,label={[labelsty]0:$v$}] (v) at (2,2.5) {};
			\node[ifvertex,label={[labelsty]180:$\sigma v$}] (sv) at (-2,2.5) {};

			\begin{scope}[opacity=0.15,transparency group]
				\node[fvertex] (sbw1) at (-1.5,-0.5) {};
				\node[fvertex] (sbw2) at ($(sbw1)+(10:0.6)$) {};
				\node[fvertex] (sbw3) at (0.2,0.4) {};
				\node[fvertex] (sbw4) at ($(sbw3)+(30:0.6)$) {};
				\node[fvertex] (sbw5) at (1.2,0.6) {};
				\node[fvertex] (sbw6) at ($(sbw5)+(-50:0.6)$) {};

				\draw[dedge,colR] (sbw1)--(sbw2);
				\draw[dedge,colR] (sbw3)--(sbw4);
				\draw[dedge,colR] (sbw5)--(sbw6);
				\draw[path] (su) to[out=-90,in=190] (sbw1);
				\draw[path] (sbw2) to[out=10,in=210] node[labelsty,below=2pt] {$\sigma \bar W$} (sbw3);
				\draw[path] (sbw4) to[out=30,in=130] (sbw5);
				\draw[path] (sbw6) to[out=-50,in=180] (u);
			\end{scope}
			\begin{scope}
				\node[fvertex] (bw1) at (1.5,-0.5) {};
				\node[fvertex] (bw2) at ($(bw1)+(170:0.6)$) {};
				\node[fvertex] (bw3) at (-0.2,0.4) {};
				\node[fvertex] (bw4) at ($(bw3)+(150:0.6)$) {};
				\node[fvertex] (bw5) at (-1.2,0.6) {};
				\node[fvertex] (bw6) at ($(bw5)+(230:0.6)$) {};

				\draw[dedge,colB] (bw1)--(bw2);
				\draw[dedge,colB] (bw3)--(bw4);
				\draw[dedge,colB] (bw5)--(bw6);
				\draw[path] (u) to[out=-90,in=-10] (bw1);
				\draw[path] (bw2) to[out=170,in=-30] node[labelsty,below=2pt] {$\bar W$} (bw3);
				\draw[path] (bw4) to[out=150,in=50] (bw5);
				\draw[path] (bw6) to[out=230,in=0] (su);
			\end{scope}
			\begin{scope}
				\node[fvertex] (w1) at (1.95,0.75) {};
				\node[fvertex] (w2) at ($(w1)+(140:0.6)$) {};
				\node[fvertex] (w3) at (1.4,1.75) {};
				\node[fvertex] (w4) at ($(w3)+(70:0.6)$) {};

				\draw[dedge,colR] (w1)--(w2);
				\draw[dedge,colR] (w3)--(w4);
				\draw[path] (u) to[out=90,in=-40] (w1);
				\draw[path] (w2) to[out=140,in=-110] node[labelsty,right=4pt] {$W$} (w3);
				\draw[path] (w4) to[out=70,in=180] (v);
			\end{scope}
			\begin{scope}
				\node[fvertex] (sw1) at (-1.95,0.75) {};
				\node[fvertex] (sw2) at ($(sw1)+(40:0.6)$) {};
				\node[fvertex] (sw3) at (-1.4,1.75) {};
				\node[fvertex] (sw4) at ($(sw3)+(110:0.6)$) {};

				\draw[dedge,colB] (sw1)--(sw2);
				\draw[dedge,colB] (sw3)--(sw4);
				\draw[path] (su) to[out=90,in=220] (sw1);
				\draw[path] (sw2) to[out=40,in=-70] node[labelsty,left=4pt] {$\sigma W$} (sw3);
				\draw[path] (sw4) to[out=110,in=0] (sv);
			\end{scope}
		\end{tikzpicture}
		\caption{Illustration of the paths and the vectors we are summing.}
		\label{fig:flexibleIffCartRScolouring:paths}
	\end{figure}
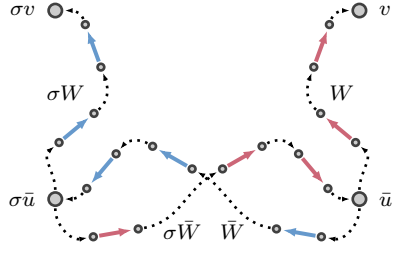

	Analogously we get
		$p_\gold(\sigma v) = p_\gold(\sigma \bar{u}) + \tau p_\gold(v)$.
	Traversing the path $\sigma \bar{W}$ from $\sigma\bar{u}$ to $\bar{u}$
	backwards from $\bar{u}$ to $\sigma\bar{u}$ yields
	\begin{align}
		\label{eq:pbluesigmabaru}
		p_\blue(\sigma \bar{u})
				&= \hspace{-0.5cm}\sum_{\substack{\oriented{\sigma u_1}{\sigma u_2} \in \sigma \bar{W}\\ \delta(\sigma u_1\sigma u_2)=\blue}}\hspace{-0.5cm} -(p(\sigma u_2)-p(\sigma u_1))
				= \hspace{-0.3cm}\sum_{\substack{\oriented{u_1}{u_2} \in \bar{W}\\ \delta(u_1u_2)=\red}}\hspace{-0.3cm} -\tau(p(u_2)-p(u_1))
				= -\tau p_\red(\sigma \bar{u})\,.
	\end{align}
	Similarly, $p_\gold(\sigma \bar{u}) = -\tau p_\gold(\sigma \bar{u})$.
	Let
	\begin{equation*}
		a(v) := p_\blue(v) - \frac{1}{2}p_\blue(\sigma \bar{u}) \qquad \text{ and } \qquad 
		z(v) := p_\gold(v) - \frac{1}{2}p_\gold(\sigma \bar{u})\,.
	\end{equation*}
	Notice that if $\bar{u}$ is invariant, the second terms in $a(v)$ and $z(v)$ are zero.
	Now we get
	\begin{align*}
		a(\sigma v) &= p_\blue(\sigma v) - \frac{1}{2}p_\blue(\sigma \bar{u})
				\eqwithreference{\eqref{eq:pbluesigmav}} p_\blue(\sigma \bar{u}) + \tau p_\red(v)  - \frac{1}{2}p_\blue(\sigma \bar{u}) \\
				&= \frac{1}{2}p_\blue(\sigma \bar{u}) + \tau p_\red(v)
				\eqwithreference{\eqref{eq:pbluesigmabaru}} -\frac{1}{2}\tau p_\red(\sigma \bar{u}) + \tau p_\red(v)
				= \tau\left( p_\red(v)-\frac{1}{2} p_\red(\sigma \bar{u})\right)\,.
	\end{align*}
	Similarly, $z(\sigma v)=\tau z(v)$.
	We define a flex
	\begin{equation*}
		p_t(v) = \rott a(v) + \rotmt\tau a(\sigma v) + z(v)\,.
	\end{equation*}
	This is the same construction as in \Cref{thm:noAlmostRBcycle2flex} which is clear
	once we show that if $u$ and $v$ are connected by a gold-red, resp.\ red-blue,
	path $(u=u_1,\ldots,u_k=v)$,
	then $a(u)=a(v)$, resp.\ $z(u)=z(v)$.
	This follows from the fact that $p_\blue(u_i)=p_\blue(u_{i+1})$ if $u_iu_{i+1}$ is a gold or red edge
	(by walk-independence, we can choose a walk $W$ from $\bar{u_i}$ to $u_{i+1}$ that passes $u_i$ as the second last vertex).
	Analogously $p_\gold(u_i)=p_\gold(u_{i+1})$ for edges $u_iu_{i+1}$ in a red-blue path.
	Now the $a$ and $z$ are the same as in \Cref{thm:noAlmostRBcycle2flex}.
	Therefore, $p_t$ is a non-trivial reflection-symmetric flex.
	We only need to show that it starts at $p$.
	Let $W$ and $\bar{W}$ be walks from $\bar{u}$ to $v$ and $\sigma\bar{u}$ respectively. Then
	\begin{align*}
		p_0(v) &= a(v) + \tau a(\sigma v) + z(v) \\
			&= p_\blue(v) - \frac{1}{2}p_\blue(\sigma \bar{u})
				 + p_\red(v)-\frac{1}{2} p_\red(\sigma \bar{u})
				 + p_\gold(v) - \frac{1}{2}p_\gold(\sigma \bar{u})\\
			&= \sum_{\oriented{u_1}{u_2} \in W} (p(u_2)-p(u_1))
				-\frac{1}{2} \sum_{\oriented{u_1}{u_2} \in \bar{W}} (p(u_2)-p(u_1))\\
			&= p(v) - p(\bar{u})  -\frac{1}{2}(p(\sigma \bar{u}) - p(\bar{u}))
			= p(v) - \frac{1}{2}(\tau p(\bar{u}) + p(\bar{u}))
			= p(v)\,,
	\end{align*}
	where the last equality follows from the assumption that $p(\bar{u})$ is on the $x$-axis.
\end{proof}

In~\cite{Bracing}, so-called (braced) P-frameworks are defined.
They consist of parallelograms with some added diagonals (braces)
and are walk-independent by \cite[Lemma~1.17]{TPframe}.
The only triangles that occur come from the bracing,
while for TP-frameworks defined in~\cite{TPframe},
there might be also triangles that are not diagonals of a 4-cycle, see \Cref{fig:Penrose}.
As both P-frameworks and TP-frameworks are walk-independent, we have the following corollary.
\begin{corollary}
	If $(G,p)$ is a TP-framework or braced P-framework that is reflection-symmetric,
	then the following are equivalent:
	\begin{enumerate}
	  \item $(G,p)$ is reflection-symmetric flexible,
	  \item $G$ has a Cartesian RS-colouring,
	  \item $G$ has a non-invariant angle-preserving class.
	\end{enumerate}
\end{corollary}
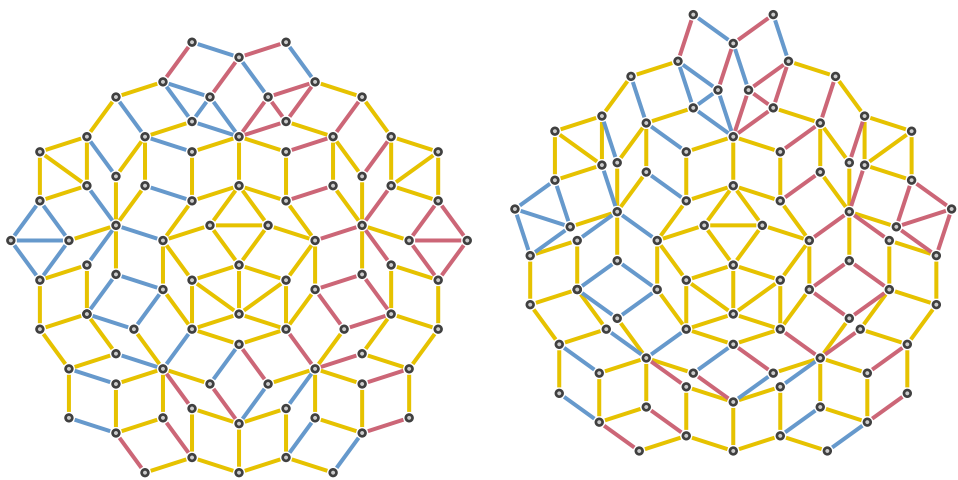
\begin{figure}[ht]
	\centering
	\begin{tikzpicture}[scale=0.65]
		\foreach \w [count=\i] in {18,36}
		{
		\begin{scope}[xshift=10*\i cm]
			\coordinate (o) at (0,0);
			\node[fvertex] (70) at (0,0) {};
			\node[fvertex] (23) at (-90:1) {};
			\node[fvertex] (20) at (-18:1) {};
			\node[fvertex] (24) at (72-18:1) {};
			\node[fvertex] (22) at (2*72-18:1) {};
			\node[fvertex] (21) at (3*72-18:1) {};
			\node[fvertex] (78) at ($(22)+(54:1)$) {};
			\foreach \x [count=\i,remember=\xo (initially 78)] in {51,76,50,77,52,79,54,80,53}
			{
				\node[fvertex,rotate around={36*\i:(o)}] (\x) at (78) {};
			}
			\node[fvertex] (82) at ($(78)+(90:1)$) {};
			\node[fvertex] (11) at ($(51)+(90:1)$) {};
			\node[fvertex] (17) at ($(53)+(90:1)$) {};

			\foreach \x/\y in {38/82,75/17,19/53,85/80,15/54}
			{
				\node[fvertex] (\x) at ($(\y)+(\w:1)$) {};
			}
			\foreach \x/\y in {14/54,84/79,18/52}
			{
				\node[fvertex] (\x) at ($(\y)+(\w-72:1)$) {};
			}
			\node[fvertex] (28) at ($(82)+(\w+36:1)$) {};
			\node[fvertex] (58) at ($(28)+(\w:1)$) {};
			\node[fvertex] (44) at ($(75)+(85)-(19)$) {};
			\node[fvertex] (42) at ($(85)+(14)-(54)$) {};
			\node[fvertex] (73) at ($(15)+(14)-(54)$) {};
			\node[fvertex] (41) at ($(73)+(84)-(14)$) {};
			\foreach \x/\y in {49/75,64/44,34/85}
			{
				\node[fvertex] (\x) at ($(\y)+(58)-(38)$) {};
			}
			\node[fvertex] (9) at ($(64)+(20)-(70)$) {};
			\node[fvertex] (69) at ($(9)+(-90:1)$) {};
			\node[fvertex] (32) at ($(69)+(85)-(34)$) {};
			\node[fvertex] (62) at ($(42)+(32)-(85)$) {};
			\node[fvertex] (7) at ($(62)+(69)-(32)$) {};
			\node[fvertex] (47) at ($(62)+(-90:1)$) {};
			\node[fvertex] (61) at ($(47)+(41)-(73)$) {};
			\node[fvertex] (6) at ($(61)+(-90:1)$) {};
			\node[fvertex] (31) at ($(61)+(84)-(41)$) {};
			\node[fvertex] (68) at ($(31)+(-90:1)$) {};
			\node[fvertex] (33) at ($(84)+(-90:1)$) {};
			\foreach \x/\y in {36/82,72/11,13/51,83/76,12/50}
			{
				\node[fvertex] (\x) at ($(\y)+(180-\w:1)$) {};
			}
			\foreach \x/\y in {10/50,81/77,16/52,74/18,43/84,63/33,8/68,71/12,39/83}
			{
				\node[fvertex] (\x) at ($(\y)+(180+72-\w:1)$) {};
			}
			\node[fvertex] (26) at ($(82)+(180-\w-36:1)$) {};
			\node[fvertex] (66) at ($(26)+(28)-(82)$) {};
			\node[fvertex] (3) at ($(66)+(58)-(28)$) {};
			\node[fvertex] (56) at ($(26)+(36)-(82)$) {};
			\node[fvertex] (0) at ($(66)+(36)-(82)$) {};
			\node[fvertex] (40) at ($(83)+(90:1)$) {};
			\node[fvertex] (29) at ($(83)+(198:1)$) {};
			\foreach \x/\y in {46/72,60/40,30/83,67/29}
			{
				\node[fvertex] (\x) at ($(\y)+(56)-(36)$) {};
			}
			\node[fvertex] (5) at ($(60)+(198:1)$) {};
			\node[fvertex] (4) at ($(67)+(71)-(12)$) {};
			\node[fvertex] (59) at ($(29)+(71)-(12)$) {};
			\node[fvertex] (45) at ($(59)+(-90:1)$) {};
			\node[fvertex] (71) at ($(39)+(-90:1)$) {};
			\node[fvertex] (55) at ($(45)+(81)-(10)$) {};
			\node[fvertex] (35) at ($(71)+(81)-(10)$) {};
			\node[fvertex] (1) at ($(55)+(-90:1)$) {};
			\node[fvertex] (25) at ($(81)+(55)-(35)$) {};
			\node[fvertex] (65) at ($(25)+(-90:1)$) {};
			\node[fvertex] (27) at ($(81)+(-90:1)$) {};
			\node[fvertex] (37) at ($(81)+(84)-(79)$) {};
			\node[fvertex] (48) at ($(74)+(-90:1)$) {};
			\node[fvertex] (57) at ($(37)+(-90:1)$) {};
			\node[fvertex] (2) at ($(65)+(57)-(27)$) {};

			\foreach \a/\b in { 48/63,74/43,18/84,52/79,77/23,48/57,74/37,16/81,52/77,79/23,20/23,49/58,75/38,17/82,53/78,24/80,20/70,21/23,62/47,73/42,85/15,80/54,20/79,70/23,59/45,71/39,83/12,50/76,77/21,21/70,76/22,51/78,56/46,72/36,82/11,60/46,40/72,83/13,51/76,78/22,61/47,73/41,84/14,54/79,80/20,24/70,64/49,75/44,19/85,80/53,24/78,24/22,70/22,76/21,50/77,45/55,35/71,81/10,61/6,68/31,33/84,43/63,48/74,57/37,81/27,65/25,1/55,5/30,9/34,8/63,33/68,84/31,41/61,73/47,42/62,32/85,34/69,64/9,9/69,64/34,44/85,75/19,17/53,82/78,51/11,72/13,40/83,60/30,67/5,60/5,67/30,83/29,59/39,45/71,35/55,81/25,65/27,57/2}
			{
				\draw[edge,colG] (\a)--(\b);
			}
			\foreach \a/\b in {26/36,58/3,66/28,4/29,8/68,33/63,43/84,18/74,16/52,81/77,10/50,12/71,83/39,59/29,67/4,59/4,67/29,83/30,40/60,72/46,65/1,25/55,81/35,10/71,50/12,83/76,51/13,72/11,82/36,56/36,26/82,56/26,0/66}
			{
				\draw[edge,colB] (\a)--(\b);
			}
			\foreach \a/\b in{28/38,32/7,62/7,69/7,65/2,57/27,81/37,16/74,18/52,84/79,54/14,73/15,42/85,32/62,32/69,34/85,64/44,49/75,58/38,68/6,61/31,41/84,73/14,54/15,80/85,19/53,17/75,82/38,66/3,58/28,82/28,26/66,0/56}
			{
				\draw[edge,colR] (\a)--(\b);
			}
		\end{scope}
		}
	\end{tikzpicture}
	\caption{A Cartesian RS-colouring of the TP-framework/braced P-framework
	obtained from a piece of a Penrose tiling.
	The flex constructed in \Cref{thm:flexibleIffCartRScolouring} for this RS-colouring is indicated.}
	\label{fig:Penrose}
\end{figure}

\section{Conclusion}
\label{sec:conclusion}

We have shown that the existence of an RS-colouring is a necessary condition
for the existence of a reflection-symmetric flex.
While in the general or rotation-symmetric case,
the existence of a (rotation-symmetric) NAC-colouring is also a sufficient condition for the existence of a (rotation-symmetric) flex,
the situation in the reflection-symmetric case is significantly more difficult.
The reason is that for a NAC-colouring $\delta$, there exists an algebraic motion
such that its active NAC-colourings are exactly $\delta$ and its conjugate.
On the other hand, the graph in \Cref{fig:RScolourings} has no RS-colouring without an almost red-blue cycle.
Hence, each of its motions (for instance the one constructed using \Cref{thm:twoRS2flex}) has at least two non-conjugated active RS-colourings.
Even worse, the following is true.
\begin{lemma}
	For every $\ell\in\NN$, there is a reflection-symmetric graph $G$
	such that every reflection-symmetric motion of $G$ has at least $\ell$ active RS-colourings.
\end{lemma}
\begin{proof}
	For $k\in\NN^+$, let $G_k$ be the reflection-symmetric graph with the vertex set
	\begin{equation*}
		\{l_0, \ldots, l_k\} \cup \{r_0, \ldots, r_k\} \cup \{m_1, \ldots, m_k\}\,,
	\end{equation*}
	the set of edges
	\begin{equation*}
		\{l_0r_0\} \cup \{l_0 l_i \colon i\in\widehat{k}\}
				\cup \{l_i m_i \colon i\in\widehat{k}\} 
				\cup \{r_0 r_i \colon i\in\widehat{k}\}
				\cup \{r_i m_i \colon i\in\widehat{k}\}
				\cup \{m_im_j \colon i,j\in\widehat{k}\},
	\end{equation*}
	where $\widehat{k} = \{1, \ldots k\}$,
	and the symmetry $\sigma l_i=r_i$ and $\sigma m_i=m_i$ for all $i$.
	For $G_k$, there are three types of RS-colourings, see \Cref{fig:kActiveRS}.
	For any reflection-symmetric motion of the graph, the angle at $r_1$ has to change.
	Hence, at least one RS-colouring of the third type must be active.
	But since each of the $k$ invariant almost red-blue 5-cycles
	has an active certificate by \Cref{lem:activeIsRS},
	at least $\lceil\frac{k}{2}\rceil$ RS-colourings of the first or second type must be active as well.
\end{proof}
\begin{figure}[ht]
	\centering
	\begin{tikzpicture}
		\draw[sym] (0,-1)--(0,2.25);
		\node[gvertex,label={[labelsty, below right]$r_0\vphantom{l}$}] (sr) at (18-72:1) {};
		\node[gvertex,label={[labelsty, below left]$l_0$}] (sl) at (18-2*72:1) {};
		\node[gvertex] (2r1) at (18:1) {};
		\node[gvertex] (t1) at (90:1) {};
		\node[gvertex] (2l1) at (90+72:1) {};
		\node[gvertex] (2r2) at (18:1.5) {};
		\node[gvertex] (t2) at (90:1.5) {};
		\node[gvertex] (2l2) at (90+72:1.5) {};
		\node[gvertex,label={[labelsty]0:$r_1$}] (2r3) at (18:2) {};
		\node[gvertex,label={[labelsty, right]$m_1$}] (t3) at (90:2) {};
		\node[gvertex,label={[labelsty]180:$l_1$}] (2l3) at (90+72:2) {};

		\draw[edge,colG] (sr)--(sl) (t1)--(t2) (t2)--(t3);
		\draw[edge,colR] (sr)--(2r1) (2r1)--(t1);
		\draw[edge,colB] (sl)--(2l1) (2l1)--(t1);
		\draw[edge,colG] (sr)--(2r2) (2r2)--(t2);
		\draw[edge,colG] (sl)--(2l2) (2l2)--(t2);
		\draw[edge,colG] (sr)--(2r3) (2r3)--(t3);
		\draw[edge,colG] (sl)--(2l3) (2l3)--(t3);
	\end{tikzpicture}
	\qquad
	\begin{tikzpicture}
		\draw[sym] (0,-1)--(0,2.25);
		\node[gvertex] (sr) at (18-72:1) {};
		\node[gvertex] (sl) at (18-2*72:1) {};
		\node[gvertex] (2r1) at (18:1) {};
		\node[gvertex] (t1) at (90:1) {};
		\node[gvertex] (2l1) at (90+72:1) {};
		\node[gvertex] (2r2) at (18:1.5) {};
		\node[gvertex] (t2) at (90:1.5) {};
		\node[gvertex] (2l2) at (90+72:1.5) {};
		\node[gvertex] (2r3) at (18:2) {};
		\node[gvertex] (t3) at (90:2) {};
		\node[gvertex] (2l3) at (90+72:2) {};

		\draw[edge,colG] (sr)--(sl) (t1)--(t2) (t2)--(t3);
		\draw[edge,colR] (sr)--(2r1) (2r1)--(t1);
		\draw[edge,colB] (sl)--(2l1) (2l1)--(t1);
		\draw[edge,colB] (sr)--(2r2) (2r2)--(t2);
		\draw[edge,colR] (sl)--(2l2) (2l2)--(t2);
		\draw[edge,colG] (sr)--(2r3) (2r3)--(t3);
		\draw[edge,colG] (sl)--(2l3) (2l3)--(t3);
	\end{tikzpicture}
	\qquad
	\begin{tikzpicture}
		\draw[sym] (0,-1)--(0,2.25);
		\node[gvertex] (sr) at (18-72:1) {};
		\node[gvertex] (sl) at (18-2*72:1) {};
		\node[gvertex] (2r1) at (18:1) {};
		\node[gvertex] (t1) at (90:1) {};
		\node[gvertex] (2l1) at (90+72:1) {};
		\node[gvertex] (2r2) at (18:1.5) {};
		\node[gvertex] (t2) at (90:1.5) {};
		\node[gvertex] (2l2) at (90+72:1.5) {};
		\node[gvertex] (2r3) at (18:2) {};
		\node[gvertex] (t3) at (90:2) {};
		\node[gvertex] (2l3) at (90+72:2) {};

		\draw[edge,colG] (sr)--(sl) (t1)--(t2) (t2)--(t3);
		\draw[edge,colB] (sr)--(2r1) (2l1)--(t1);
		\draw[edge,colR] (sl)--(2l1) (2r1)--(t1);
		\draw[edge,colR] (sr)--(2r2) (2l2)--(t2);
		\draw[edge,colB] (sl)--(2l2) (2r2)--(t2);
		\draw[edge,colB] (sr)--(2r3) (2l3)--(t3);
		\draw[edge,colR] (sl)--(2l3) (2r3)--(t3);
	\end{tikzpicture}
	\caption{Graphs $G_k$ (here $G_3$ is shown) have three types of RS-colourings:
	one of the $k$ invariant 5-cycles is almost red-blue with incident edges
	having the same colour (left), two of the $k$ invariant 5-cycles coloured as in the first type (middle)
	and each of the 5-cycles coloured with red and blue interchanging (right).}
	\label{fig:kActiveRS}
\end{figure}
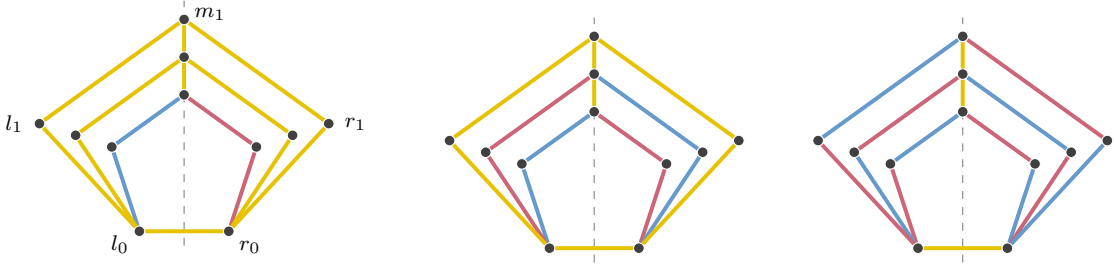
 
The class of graphs $G_k$ from the lemma above indicates that a construction of a flex for a graph with $G$ as a subgraph
should employ at least $\lceil\frac{k}{2}\rceil$ RS-colourings.

Interestingly,
a similar situation to the reflection case is observed for periodic symmetry~\cite{Dperiodic2019}.
It is shown there that certain periodic NAC-colourings called flexible 2-lattice NBAC-colour\-ings are necessary for flexibility, but the author could not prove they were sufficient for~it.

\addcontentsline{toc}{section}{Acknowledgments}
\section*{Acknowledgements}
S.\,D.\ was supported by the Heilbronn Institute for Mathematical Research.
G.\,G.\ was partially supported by the Austrian Science Fund (FWF): 10.55776/I6233.
J.\,L.\ was funded by the Czech Science Foundation (GAČR) project 22-04381L.

This research was funded in whole, or in part, by the Austrian Science Fund (FWF) 10.55776/I6233. For the purpose of open access, the authors have applied a CC BY public copyright licence to any author accepted manuscript version arising from this submission.

\bibliography{symmetric}

\end{document}